\documentclass[11pt]{amsart}
\usepackage{amssymb,amsmath}
\usepackage[dvips]{graphicx}
\usepackage{amsmath,amscd}
\usepackage{amsthm}
\usepackage{subcaption}
\usepackage{verbatim}
\usepackage{MnSymbol}

 \usepackage[usenames,dvipsnames]{pstricks}
 \usepackage{epsfig}
 \usepackage{pst-grad} 
 \usepackage{pst-plot} 

\theoremstyle{plain} \newtheorem{thm}{Theorem}[section]
\newtheorem{cor}[thm]{Corollary} \newtheorem{prop}[thm]{Proposition}
\newtheorem{lemma}[thm]{Lemma}

\newtheorem*{namedtheorem}{\theoremname}
\newcommand{\theoremname}{testing}

\theoremstyle{definition} \newtheorem{defn}[thm]{Definition}

\theoremstyle{remark} \newtheorem{remark}[thm]{Remark}

\newcommand{\h}{\mathbb{H}}

\newcommand{\R}{\mathbb{R}}
\newcommand{\s}{\mathbb{S}}

\newcommand{\Z}{\mathbb{Z}}

\begin{document}

\title{Hidden Symmetries and Commensurability of 2-Bridge Link Complements}

\author{Christian Millichap \& William Worden} 
\address{Department of Mathematics, Linfield College\\ McMinnville, OR 97128}
\email{Cmillich@linfield.edu}
\address{Department of Mathematics, Temple University\\ Philadelphia, PA 19122}
\email{william.worden@temple.edu}

\begin{abstract}
In this paper, we show that any non-arithmetic hyperbolic $2$-bridge link complement admits no hidden symmetries.  As a corollary, we conclude that a hyperbolic $2$-bridge link complement cannot irregularly cover a hyperbolic $3$-manifold. By combining this corollary with the work of Boileau and Weidmann, we obtain a characterization of $3$-manifolds with non-trivial JSJ-decomposition and rank two fundamental groups. We also show that the only commensurable hyperbolic $2$-bridge link complements are the figure-eight knot complement and the $6_{2}^{2}$ link complement. Our work requires a careful analysis of the tilings of $\mathbb{R}^{2}$ that come from lifting the canonical triangulations of the cusps of hyperbolic $2$-bridge link complements. 
\end{abstract}

\maketitle
\section{Introduction}
\label{sec:intro}
Two manifolds are called \textit{commensurable} if they share a common finite sheeted cover.  Here, we focus on hyperbolic $3$-manifolds, that is, $M = \mathbb{H}^{3} / \Gamma$ where $\Gamma$ is a discrete, torsion-free subgroup of $\text{Isom}(\mathbb{H}^{3})$.  We are interested in analyzing the set of all manifolds commensurable with $M$.  Commensurability is a property of interest because it provides a method for organizing manifolds and many topological properties are preserved within a commensurability class. For instance,  Schwartz \cite{Sch} showed that two cusped hyperbolic $3$-manifolds are commensurable if and only if their fundamental groups are quasi-isometric. In this paper, we restrict our attention to hyperbolic $2$-bridge link complements; see Section \ref{sec:background} for the definition of a $2$-bridge link. We use the word \textit{link} to refer to a link in $\mathbb{S}^{3}$ with at least one component. We use the word \textit{knot} to only mean a single component link. 

A significant challenge in understanding the commensurability class of a hyperbolic $3$-manifold $M = \mathbb{H}^{3} / \Gamma$ is determining whether or not $M$ has any \textit{hidden symmetries}. To understand hidden symmetries, we first need to introduce some terminology. The \textit{commensurator} of $\Gamma$ is 
\begin{center}
	$C(\Gamma) = \left\{ g \in \text{Isom}(\mathbb{H}^{3}) : \left|\Gamma : \Gamma \cap g\Gamma g^{-1} \right| < \infty\right\rbrace $.	
\end{center}
It is a well known fact that two hyperbolic $3$-manifolds are commensurable if and only if their corresponding commensurators are conjugate in $\text{Isom}(\mathbb{H}^{3})$; see \cite[Lemma 2.3]{Wa}. We denote by $C^{+}(\Gamma)$ the restriction of $C(\Gamma)$ to orientation-preserving isometries. We also denote by $N(\Gamma)$ the normalizer of $\Gamma$ in $\text{Isom}(\mathbb{H}^{3})$ and $N^{+}(\Gamma)$ the restriction of $N(\Gamma)$ to orientation-preserving isometries. Note that, $\Gamma \subset N(\Gamma) \subset C(\Gamma)$. A symmetry of $M$ corresponds to an element of $N(\Gamma) / \Gamma$, and a hidden symmetry of $M$ corresponds to an element of $C(\Gamma)$ that is not in $N(\Gamma)$. Geometrically, $M$ admits a hidden symmetry if there exists a symmetry of a finite cover of $M$ that is a not a lift of an isometry of $M$. See Sections 2 and 3 of \cite{Wa} for more details on commensurators and hidden symmetries.      

In this paper, we give a classification of the hidden symmetries of hyperbolic $2$-bridge link complements. In \cite{ReWa} Reid--Walsh used algebraic methods to determine that hyperbolic $2$-bridge knot complements (other than the figure-eight knot complement) have no hidden symmetries. However, their techniques do not apply to hyperbolic $2$-bridge links with two components. Here, we use a geometric and combinatorial approach to prove the following theorem. 

\begin{thm}
	\label{thm:main1}
	If $M = \mathbb{S}^{3} \setminus K$ is a non-arithmetic hyperbolic $2$-bridge link complement, then $M$ admits no hidden symmetries (both orientation-preserving and orientation-reversing). 
\end{thm}

The only arithmetic hyperbolic $2$-bridge links are the figure-eight knot, the Whitehead link, the $6_{2}^{2}$ link, and the $6_{3}^{2}$ link. Though it will not be needed in what follows, we refer the interested reader to \cite[Definition 8.2.1]{MaRe} for the definition of an arithmetic group $\Gamma\le \mathrm{Isom}(\h^3)$.

We prove Theorem \ref{thm:main1} by using the canonical triangulation $\mathcal{T}$ of a hyperbolic $2$-bridge link complement, $M = \mathbb{H}^{3} / \Gamma = \mathbb{S}^{3} \setminus K$. This triangulation was first described by Sakuma--Weeks in \cite{SaWe}. Gu\'{e}ritaud in his thesis \cite{Gu} proved that this triangulation is geometrically canonical, i.e., topologically dual to the Ford--Voronoi domain for equal volume cusp neighborhoods. In addition, Akiyoshi--Sakuma--Wada--Yamashita in \cite{AkSaWaYa} have announced a proof of this result where they analyze the triangulation $\mathcal{T}$ via cone deformations of $M$ along the unknotting tunnel. Futer also showed that this triangulation is geometric by applying Rivin's volume maximization principle;  see the appendix of \cite{GuFu}. By \cite[Theorem 2.6]{GoHeHo}, if any such $M$ is non-arithmetic, then $C(\Gamma)$ can be identified with the group of symmetries of the tiling of $\mathbb{H}^{3}$ obtained by lifting $\mathcal{T}$, which we call $\widetilde{\mathcal{T}}$. We prove that any non-arithmetic hyperbolic $2$-bridge link complement $M$ does not admit hidden symmetries by showing that any symmetry of $\widetilde{\mathcal{T}}$ actually corresponds to a composition of symmetries of $M$ and deck transformations of $M$. In other words, $C(\Gamma) = N(\Gamma)$. 

Rather than analyze this tiling of $\mathbb{H}^{3}$, we drop down a dimension and instead analyze the (canonical) cusp triangulation $\widetilde{T}$ of $\mathbb{R}^{2}$, induced by $\widetilde{\mathcal{T}}$. By intersecting a cusp cross section of $M$ with its canonical triangulation $\mathcal{T}$, we obtain a canonical triangulation $T$ of the cusp(s).  If $K$ has two components, we still end up with the same canonical triangulation on both components of $T$ since there is always a symmetry exchanging the two components, and we take equal volume cusp neighborhoods. We can lift $T$ to a triangulation $\widetilde{T}$ of $\mathbb{R}^{2}$ (or two copies of $\mathbb{R}^{2}$ if $K$ has two components). We also place edge labels on $\widetilde{T}$ which record edge valences of corresponding edges in the three-dimensional triangulation. This labeling provides us with enough rigid structure in $\widetilde{T}$ to rule out any hidden symmetries. Goodman--Heard--Hodgson use a similar approach to prove that non-arithmetic hyperbolic punctured-torus bundles do not admit hidden symmetries \cite[Theorem 3.1]{GoHeHo}.

If a hyperbolic $3$-manifold $M$ admits no hidden symmetries, then $M$ can not irregularly cover any \textit{hyperbolic $3$-orbifolds}. A hyperbolic $3$-orbifold is any $N = \mathbb{H}^{3} / \Gamma$, where $\Gamma$ is a discrete subgroup of $\text{Isom}(\mathbb{H}^{3})$, possibly with torsion. All of the previous statements about commensurability of hyperbolic $3$-manifolds and the commensurator of $\Gamma$ also hold for hyperbolic $3$-orbifolds. Theorem \ref{thm:main1} quickly gives us the following corollary about coverings of hyperbolic $3$-orbifolds by hyperbolic $2$-bridge link complements. For the arithmetic cases, volume bounds are taken into consideration to rule out irregular covers of manifolds.  

\begin{cor}
\label{cor:main1}
Let $M$ be any hyperbolic $2$-bridge link complement. If $M$ is non-arithmetic, then $M$ does not irregularly cover any hyperbolic $3$-orbifolds (orientable or non-orientable). If $M$ is arithmetic, then $M$ does not irregulary cover any (orientable) hyperbolic $3$-manifolds. 
\end{cor}

By combining Corollary \ref{cor:main1} with the work of Boileau--Weidmann in \cite{BoWe}, we get the following characterization of $3$-manifolds with non-trivial JSJ-decomposition and rank two fundamental groups. For a more detailed description of this decomposition see Section \ref{subsec:irregular}. 

\begin{cor}
\label{cor:main2}
Let $M$ be a compact, orientable, irreducible $3$-manifold with rank($\pi_{1}(M)) = 2$. If $M$ has a non-trivial JSJ-decomposition, then one of the following holds:

1. $M$ has Heegaard genus $2$.

2. $M$ decomposes into a Seifert fibered $3$-manifold and hyperbolic $3$-manifold.

3. $M$ decomposes into two Seifert fibered $3$-manifolds.
\end{cor}

The original characterization given by Bolieau--Weidmann included a fourth a possiblity: a hyperbolic piece of $M$ is irregularly covered by a $2$-bridge link complement. Corollary \ref{cor:main1} eliminates this possibility.

Ruling out hidden symmetries also plays an important role in analyzing the \textit{commensurability class} of a hyperbolic $3$-orbifold $M = \mathbb{H}^{3} / \Gamma$.  By the commensurability class of a hyperbolic $3$-orbifold (or manifold) $N$, we mean the set of all hyperbolic $3$-orbifolds commensurable with $N$. A fundamental result of Margulis \cite{Mar} implies that $C(\Gamma)$ is discrete in $\text{Isom}(\mathbb{H}^{3})$ (and $\Gamma$ is finite index in $C(\Gamma)$) if and only if $\Gamma$ is non-arithmetic. Thus, in the arithmetic case, $M$ will have infinitely many hidden symmetries. In the non-arithmetic case, this result implies that the hyperbolic $3$-orbifold $\mathcal{O}^+ = \mathbb{H}^{3} / C^{+}(\Gamma)$ is the unique minimal (orientable) orbifold in the commensurability class of $M$. So, in the non-arithmetic case, $M$ and $M'$ are commensurable if and only if they cover a common minimal orbifold. Furthermore, when $M$ admits no hidden symmetries, $C^{+}(\Gamma) = N^{+}(\Gamma)$, and so, $\mathcal{O}^+$ is just the quotient of $M$ by its orientation-preserving symmetries. 

By using Theorem \ref{thm:main1} and thinking about commensurability in terms of covering a common minimal orbifold, we obtain the following result about commensurability classes of hyperbolic $2$-bridge link complements. 



\begin{thm}
	\label{thm:main2}
	The only pair of commensurable hyperbolic $2$-bridge link complements are the figure-eight knot complement and the $6_{2}^{2}$ link complement. 
\end{thm}

We prove Theorem \ref{thm:main2} by analyzing the cusp of each minimal (orientable) orbifold, $\mathcal{O}^+$, in the commensurability class of a non-arithmetic hyperbolic $2$-bridge link complement. This orbifold always has one cusp since two component $2$-bridge links always have a symmetry exchanging the components. The cusp of this orbifold inherits a canonical cellulation from the canonical triangulation $T$ of the cusp(s) of $M$. By comparing minimal orbifold cusp cellulations, we establish this result. 

We now describe the organization of this paper. In Section \ref{sec:background}, we provide some background on $2$-bridge links, including an algorithm for building any $2$-bridge link from a word $\Omega$ in Ls and Rs. Section \ref{sec:CuspT} describes how to build the canonical triangulation of a $2$-bridge link complement and the corresponding cusp triangulation $T$ based on this word $\Omega$. In this section we also prove some essential combinatorial properties of $\widetilde{T}$, the lift of $T$ to $\mathbb{R}^{2}$. Section \ref{sec:sym} analyzes the possible symmetries of a $2$-bridge link complement in terms of the word $\Omega$, and describes the actions of these symmetries on $\widetilde{T}$. In Section \ref{sec:hiddensym}, we prove Theorem \ref{thm:main1}, Corollary \ref{cor:main1}, and Corollary \ref{cor:main2}. In Section \ref{sec:comm}, we prove Theorem \ref{thm:main2}.

The authors would like to thank David Futer for helpful conversations and guidance on this work. We would also like to thank the referee for making a number of helpful suggestions.


\section{Background on $2$-bridge links}
\label{sec:background}

In order to describe $2$-bridge links, we first need to define \textit{rational tangles}. First, a \textit{$2$-tangle} is a pair ($B$, $t$), where $t$ is a pair of unoriented arcs embedded in the $3$-ball $B$ so that $t$ only intersects the boundary of $B$ in four specified marked points: SW, SE, NW, and NE (if we think of $\partial B$ as the unit sphere centered at the origin in $\mathbb{R}^{3}$, then SW is the southwest corner $(\frac{-1}{\sqrt{2}},\frac{-1}{\sqrt{2}},0)$, SE is the southeast corner $(\frac{1}{\sqrt{2}},\frac{-1}{\sqrt{2}},0)$, etc). Rational tangles are a special class of $2$-tangles. The simplest rational tangles are the $0$-tangle and the $\infty$-tangle. The $0$-tangle consists of two arcs that don't twist about one another, with one arc connecting NW to NE, and the other arc connecting SW to SE. Similarly, the $\infty$-tangle consists of two unknotted arcs, with one arc connecting NE to SE and the other arc connecting NW to SW. Both of these tangles admit an obvious meridian curve contained on $\partial B$ that bounds an embedded disk in the interior of $B$. A rational tangle is constructed by taking one of these trivial tangles and alternating between twisting about the western endpoints (NW and SW) and twisting about the southern endpoints (SW and SE). This twisting process maps the meridian of the $0$-tangle ($\infty$-tangle) to a closed curve with rational slope $\frac{p}{q}$, which determines this tangle, hence the name rational tangle. A \textit{$2$-bridge link} is constructed by taking a rational tangle, connecting its western endpoints by an unknotted strand, and connecting its eastern endpoints by an unknotted strand. 


Here, we describe a $2$-bridge link $K \subset \mathbb{S}^{3}$ in terms of a word $\Omega$, which is a sequence of Ls and Rs: $\Omega = R^{\alpha_{1}}L^{\alpha_{2}}R^{\alpha_{3}}\cdots R^{\alpha_{n}}$, $\alpha_{i} \in \mathbb{N}$ (if $n$ is odd and the starting letter is $R$). The sequence $\left[\alpha_{1}+1, \alpha_{2}, \dots, \alpha_{n-1}, \alpha_{n}+1\right]$ gives the continued fraction expansion for the rational tangle $\frac{p}{q}$ used to construct a $2$-bridge link. Each $L$ corresponds to performing a left-handed half-twist about the NW and SW endpoints of a $0$-tangle and each $R$ corresponds to performing a right-handed half-twist about the SW and SE endpoints of an $\infty$-tangle. Each \textit{syllable}, i.e., each maximal subword $L^{\alpha_{i}}$ or $R^{\alpha_{i}}$, corresponds to two strands wrapping around each other $\alpha_{i}$ times. This word $\Omega$ gives a procedure to construct an alternating $4$-string braid between two $4$-punctured spheres, $S_{1}$ and $S_{c}$, where $S_{1}$ is exterior to the braid and $S_{c}$ is interior to the braid; see Figure \ref{fig:ex_link}. To construct a $2$-bridge link, we add a single crossing to the outside of $S_{1}$, and we add a single crossing to the inside of $S_{c}$. There is a unique way to add these crossings so that the resulting link diagram is alternating. Any $2$-bridge link can be constructed in this manner and we use the notation $K(\Omega)$ to designate the $2$-bridge link constructed by the word $\Omega$. The original source for this notation comes from the appendix of \cite{GuFu}, which contains more details of this construction. 

\begin{figure}
        \centering
        \begin{subfigure}[b]{0.4\textwidth}
                \includegraphics[scale=0.6]{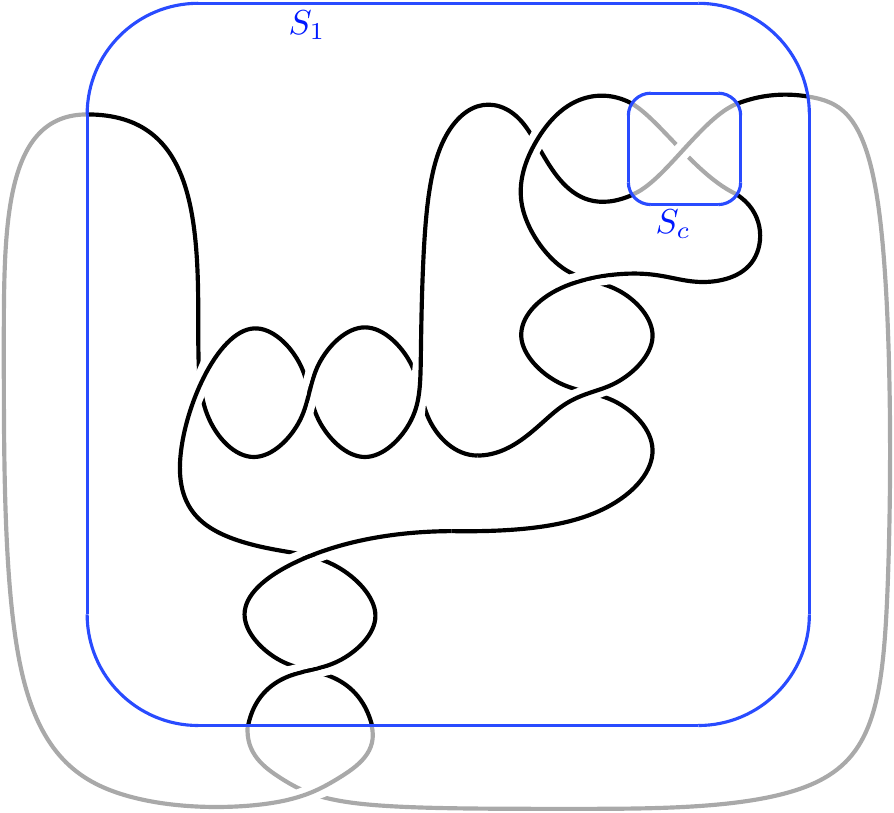}
                \caption{}
                \label{fig:ex_linkA}
        \end{subfigure}
       \qquad\qquad
        \begin{subfigure}[b]{0.4\textwidth}
                \includegraphics[scale=0.6]{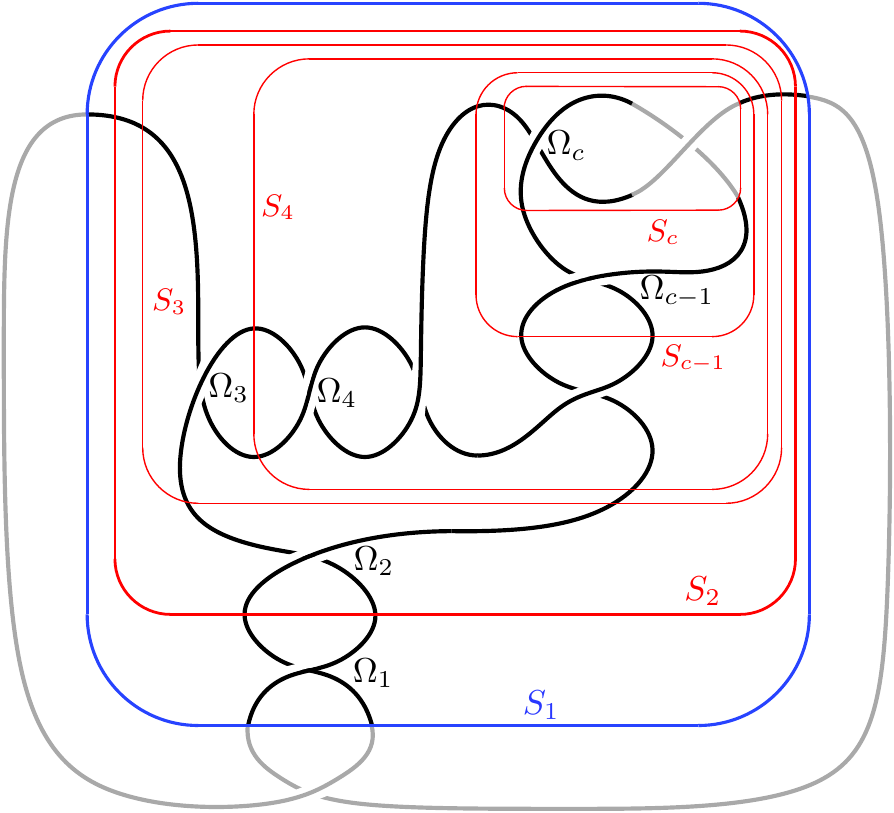}
                \caption{}
                \label{fig:ex_linkB}
        \end{subfigure}       
        \caption{Left: The link $K(\Omega)$, where $\Omega=R^2L^3R^2L$, read from $S_1$ inward to $S_c$. On the right is the same link, with crossings labelled and $4$-punctured spheres $S_i$ shown (note that $S_5$ and $S_6$ are omitted for readability).}
        \label{fig:ex_link}
\end{figure}

The following are important facts about $2$-bridge links that we will use. From now on, we will state results in terms of $K(\Omega)$ and we assume that any $2$-bridge link has been constructed in the manner described above, unless otherwise noted.
\begin{itemize}
\item Given a $2$-bridge link $K(\Omega)$, we obtain a mirror image of the same link (with orientations changed on $\mathbb{S}^{3}$) if we switch Ls and Rs in the word $\Omega$. Since we will only be considering unoriented link complements, we consider such links equivalent. 
\item $2$-bridge links (and their complements) are determined by the sequence of integers $\alpha_{1}, \dots, \alpha_{n}$ up to inversion. Schubert gives this classification of $2$-bridge knots and links in \cite{Schu}, and Sakuma--Weeks \cite[Theorem II.3.1]{SaWe} give this classification of their complements by examining their (now known) canonical triangulations.
\item A $2$-bridge link $K(\Omega)$ is hyperbolic if and only if $\Omega$ has at least two syllables. This follows from Menasco's classification of alternating link complements \cite{Men}. 
\item The \textbf{only} arithmetic hyperbolic $2$-bridge links are those listed below. This classification was given by Gehring--Maclachlan--Martin in \cite{GeMaMa}. 
\begin{itemize}
\item The figure-eight knot given by $RL$ or $LR$,
\item The Whitehead link given by $RLR$ or $LRL$,
\item The $6_{2}^{2}$ link given by $L^{2}R^{2}$ or $R^{2}L^{2}$, and
\item The $6_{3}^{2}$ link given by $RL^{2}R$ or $LR^{2}L$.
\end{itemize} 
We care about distinguishing between non-arithmetic and arithmetic hyperbolic link complements because different techniques have to be used for analyzing hidden symmetries and commensurability classes. 
\end{itemize}

Throughout this paper, we will always assume that $K(\Omega)$ is hyperbolic, i.e., $\Omega$ has at least two syllables. In Section \ref{sec:CuspT}, we will use the diagram of $K(\Omega)$ described above to build the canonical cusp triangulation of $\mathbb{S}^{3} \setminus K(\Omega)$. 


\section{Cusp triangulations of $2$-bridge link complements}
\label{sec:CuspT}

Let $K=K(\Omega)$ be a $2$-bridge link, defined as in Section \ref{sec:background}, with $\Omega$ a word in $R$ and $L$, and $\Omega_i$ its $i^{th}$ letter. We may assume that $\Omega_1=R$, as mentioned in Section \ref{sec:background}. In this section we give a description of the construction of the triangulation $\mathcal{T}$ of $\s^3\setminus K$, and of the induced cusp triangulation $T$, and its lift $\widetilde{T}$ (if $K$ has two components, then the two cusp triangulations are identical). We then describe an algorithmic approach for constructing $\widetilde{T}$, and prove some facts about simplicial homeomorphisms $f:\widetilde{T} \to \widetilde{T}$. Our description of these triangulations follows that of \cite[Appendix A]{GuFu} and \cite[Chapter II]{SaWe}, to which we refer the reader for further details.

To build the triangulation $\mathcal{T}$, we first place a $4$-punctured sphere $S_i$ at each crossing $\Omega_i$ corresponding to a letter of $\Omega$, so that every crossing $\Omega_j$ for $j\ge i$ is on one side of $S_i$, and the remaining crossings are on the other side; see Figure \ref{fig:ex_linkB}. We will start by focusing on $S_1$ and $S_2$. We triangulate both of them as shown in the first frame of Figure \ref{fig:S1_glue_S2} (notice that the edge from the lower-left to upper-right puncture is in front for both). If we push $S_1$ along the link to the other side of the crossing $\Omega_1$, we see that some of its edges coincide with edges of $S_2$ (in particular, the horizontal edges coincide, and the diagonal edges of $S_1$ become vertical in $S_2$, see Figure \ref{fig:S1_glue_S2}). The vertical edges of $S_1$, however, get pushed to diagonal edges that cannot be identified to the diagonal edges of $S_2$. The top frame of Figure \ref{fig:tetra_layer} shows $S_1$ and $S_2$ with appropriate edges identified, as seen lifted to $\R^2\setminus \Z^2$ (i.e., cut along top, bottom, and left edges then unfold). If we lift $S_1$  to $\R^2\setminus \Z^2$ in such a way that its triangulation has edge slopes $\frac{0}{1},\frac{1}{1},\frac{1}{0}$, this choice forces $S_2$ to have edge slopes $\frac{0}{1},\frac{1}{2},\frac{1}{1}$, as shown in the lower frame of Figure \ref{fig:tetra_layer}. This means that the triangulation of $S_2$ in $\R^2\setminus \Z^2$ is obtained by applying the matrix $R=\left( \begin{smallmatrix} 1&1\\0&1\\ \end{smallmatrix}\right)$ to the $S_1$ triangulation of $\R^2\setminus \Z^2$. If the letter $\Omega_1$ between $S_1$ and $S_2$ had been an $L$, we would have found by the same analysis that the matrix taking us from the triangulation of $S_1$ to the triangulation of $S_2$ must be $L=\left(\begin{smallmatrix} 1&0\\1&1\\ \end{smallmatrix}\right)$. This holds in general. If we know the edge slopes of the triangulation of $S_i$, we can apply the appropriate matrix, depending on whether $\Omega_i$ is an $R$ or an $L$, to get the triangulation of $S_{i+1}$ (see Figure \ref{fig:RL_move}).

\begin{remark}
Though we do not use this fact in what follows, the word $\Omega$ can be viewed as a path in the Farey tesselation, with each letter corresponding to making a right (for $R$) or left (for $L$) turn from one Farey triangle to the next. In this case each four punctured sphere $S_i$ corresponds to a Farey triangle, and its slopes are given by the vertices of that triangle. For details of this approach, we again direct the interested reader to \cite{GuFu} and \cite{SaWe}.
\end{remark}

Coming back to $S_1$ and $S_2$, we see in Figure \ref{fig:tetra_layer} that between the (red) triangulation of $S_2$ and the (blue) triangulation of $S_1$ is a layer of two tetrahedra, which we denote $\Delta_1$. Similarly, between the $4$-punctured spheres $S_i$ and $S_{i+1}$ we get a layer $\Delta_i$ of tetrahedra.  This construction results in a ``product region" $S\times I$, where $S\times \{0\}=S_1$ and $S\times \{1\}=S_c$. We use quotation marks here because for $\Omega \in \{RL^k, LR^k, RL^kR, LR^kL\}$, $S\times I$ is not a true product since there will be an edge shared by all the $S_i$. 

\begin{figure}
        \centering
        \begin{subfigure}[b]{0.50\textwidth}
                \includegraphics[scale=0.45]{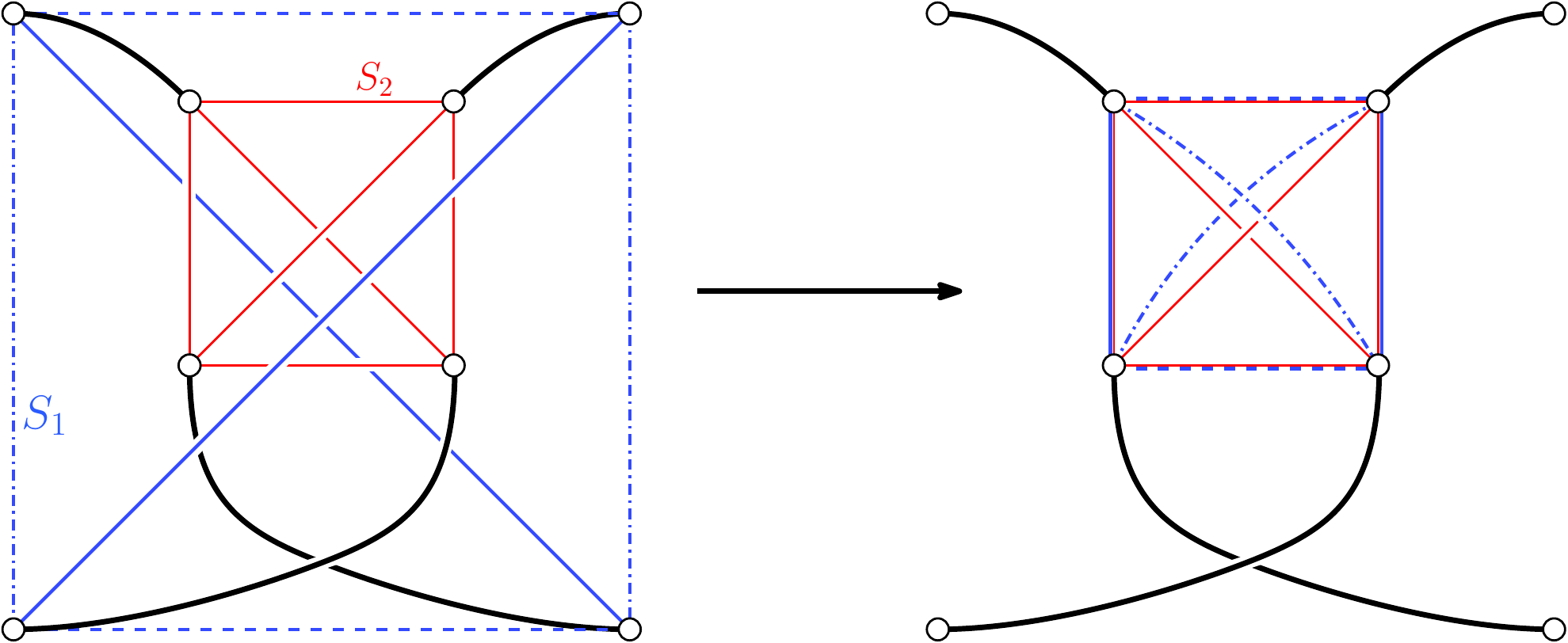}
                \caption{}
                \label{fig:S1_glue_S2}
        \end{subfigure}
       \qquad\qquad\qquad\qquad
        \begin{subfigure}[b]{0.25\textwidth}
                \includegraphics[scale=0.4]{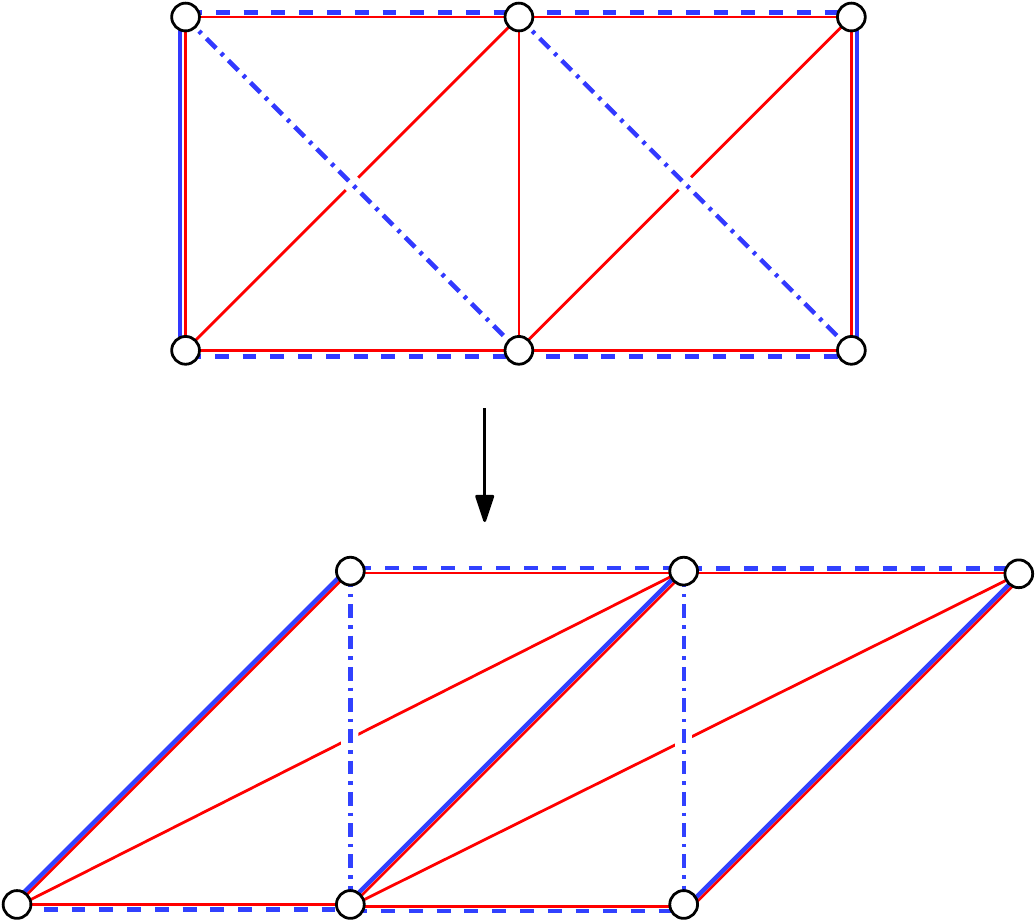}
                \caption{}
                \label{fig:tetra_layer}
        \end{subfigure}       
        \caption{On the left (A) we see which edges of $S_1$ are identified to edges of $S_2$, and what the region between $S_1$ and $S_2$ looks like. In the right figure (B) it is a little easier to see, with $S_1$ and $S_2$ unfolded, that the region between them is a pair of tetrahedra.}
        \label{fig:S1_to_S2}
\end{figure}

\begin{figure}[h]
    \centering
    \includegraphics[scale=.7]{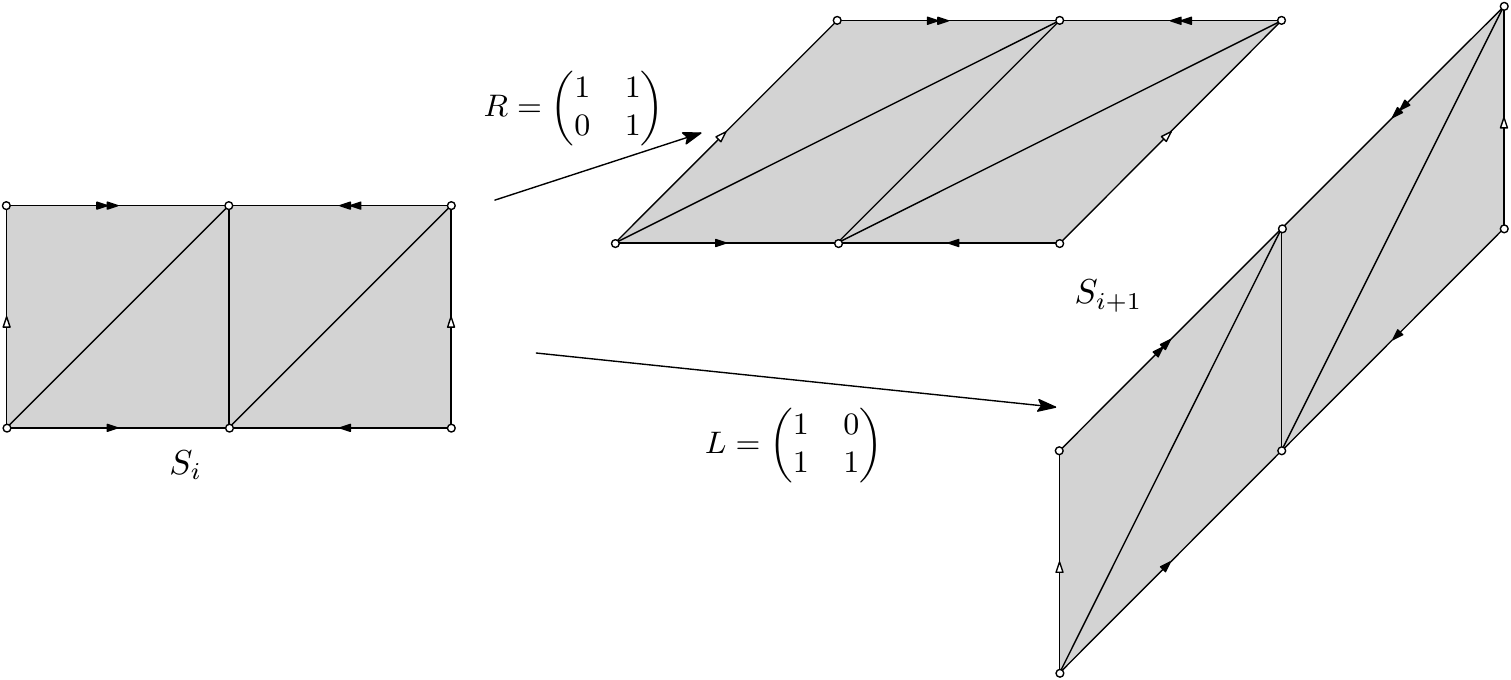}
    \caption{Attaching $S_2$ to $S_1$. In the upper right picture, the 1-skeleton of one of the tetrahedra in the layer $\Delta_1$ is shown in red.}
    \label{fig:RL_move}
\end{figure}

To obtain $\s^3\setminus K$ from $S\times I$, we first ``clasp" $S_1$ by folding along edges with slope $\frac{1}{1}$ and identifying pairs of triangles adjacent to those edges, as shown in Figure \ref{fig:clasping}. We clasp $S_c$ in the same way, this time folding along either the edge with greatest slope or the edge with least slope, depending on whether the final letter of $\Omega$ is $R$ or $L$, respectively.

\begin{figure}[h]
    \centering
    \includegraphics[scale=.8]{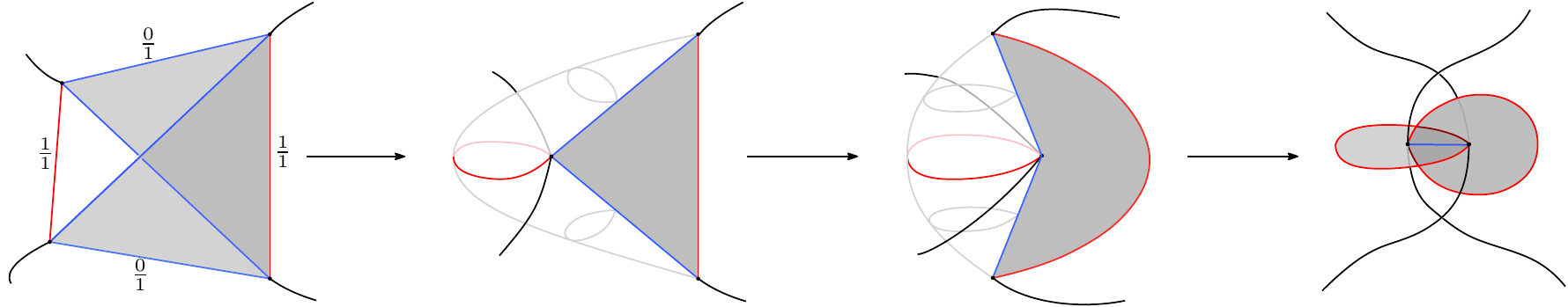}
    \caption{The clasping of $S_1$. The viewpoint of the reader is the ``inside" of $S_1$, i.e., the side containing the braid in Figure \ref{fig:ex_link}}
    \label{fig:clasping}
\end{figure}

To understand the induced triangulation $T$ of a cusp cross section, we first consider a neighborhood of a single puncture $P$ in $S\times I$. For each layer of tetrahedra $\Delta_i$ between $S_i$ and $S_{i+1}$, we get a pair of triangles $D_i$ and $D_i'$ going once around the puncture, as in Figure \ref{fig:cusp_cross1}. In this figure vertices of $D_i\cup D_i'$ are labelled according to the edges of $\Delta_i$ that they are contained in, and edges of $D_i\cup D_i'$ are labelled according to the edge of $\Delta_i$ that they are across a face from. Notice in Figure \ref{fig:cusp_cross1} that $D_i$ has a vertex ($c_-$) meeting an edge of $S_i$ but not meeting $S_{i+1}$, and $D_i'$ has a vertex ($\tilde{c}$) meeting $S_{i+1}$ but not meeting $S_i$. Thus $D_i$ is distinguished from $D_i'$.

\begin{figure}
        \centering
        \begin{subfigure}[b]{0.55\textwidth}
                \includegraphics[scale=1]{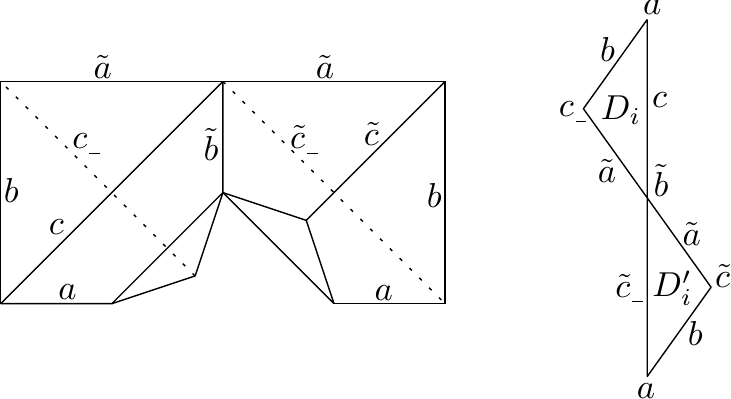}
                \caption{}
                \label{fig:cusp_cross1}
        \end{subfigure}
       \qquad\qquad
        \begin{subfigure}[b]{0.3\textwidth}
                \includegraphics[scale=1]{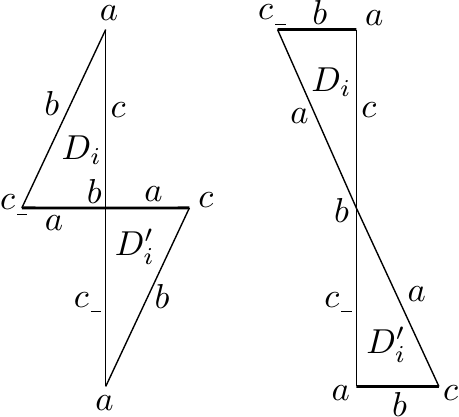}
                \caption{}
                \label{fig:EV_cor2}
        \end{subfigure}       
        \caption{(A): In (A), a layer $\Delta_i$ with a neighborhood of a cusp removed (left), and the triangles $D_i\cup D_i'$ that the layer $\Delta_i$ contributes to the cusp triangulation (right). Edges with the same slope have labels that differ by a $\sim$ decoration. Figure (B) shows $D_i$ and $D_i'$ after being adjusted as prescribed in Figure 6, with $\sim$ decorations removed so that edges with the same slope are labelled the same.}
        \label{fig:cusp_cross_EV}
\end{figure}

To see how $D_i\cup D_i'$ attaches to $D_{i-1}\cup D_{i-1}'$, we must consider how $\Delta_i$ attaches to $\Delta_{i-1}$. Figure \ref{fig:cusp_cross2a} shows $\Delta_{i}$ and $\Delta_{i-1}$ in $(\R^2\setminus \Z^2)\times I$ (sandwiched between $S_{i-1}\cup S_i\cup S_{i+1}$) in the case where $\Omega_{i}=R$, and the corresponding triangles around the puncture. There is a unique edge $e$ of $D_i\cup D_i'$, corresponding to an edge of $S_i$ shared by both $S_{i-1}$ and $S_{i+1}$, and with vertices $v_1\in S_{i-1}$ and $v_2\in S_{i+1}$. This means that the edge $e$ moves us along the cusp cross section in the longitudinal direction, so it will be part of a longitude in $\widetilde{T}$. It makes sense then to adjust these edges to be horizontal, as we build the triangulation $\widetilde{T}$ (see Figure \ref{fig:cusp_cross2a}). Figure \ref{fig:cusp_cross2b} shows the analogous adjustment when $\Omega_i=L$.

When we clasp $S_1$, an edge of $D_1$ is identified to an edge of $D_1'$, and similarly for $D_c'$ and $D_c$ when $S_c$ is clasped, as illustrated in Figure \ref{fig:cusp_trian1}. We will call the triangles $D_1$ and $D_c'$ \emph{clasping triangles}. For $\Omega=R^2L^3R^2L$, the triangulation around a puncture before clasping and after clasping is shown in Figures \ref{fig:cusp_trian1b} and \ref{fig:cusp_trian1c}, respectively.

\begin{figure}
        \centering
        \begin{subfigure}[b]{0.5\textwidth}
                \includegraphics[scale=0.8]{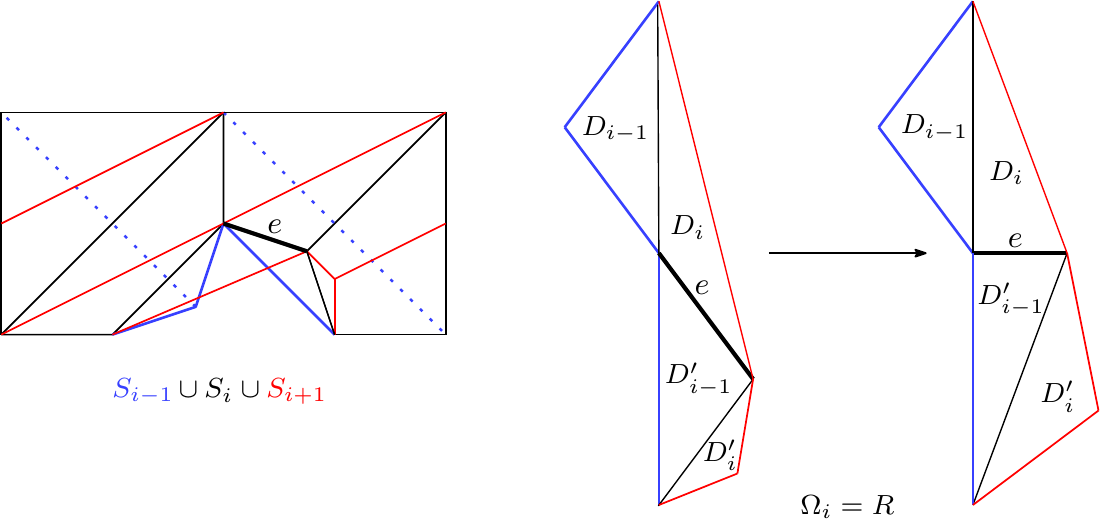}
                \caption{}
                \label{fig:cusp_cross2a}
        \end{subfigure}
       \qquad\qquad\qquad\quad
        \begin{subfigure}[b]{0.25\textwidth}
                \includegraphics[scale=0.9]{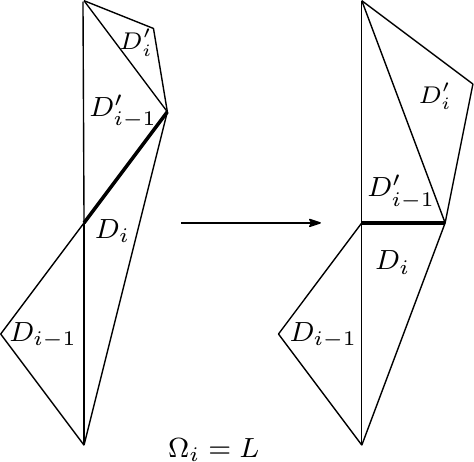}
                \caption{}
                \label{fig:cusp_cross2b}
        \end{subfigure}       
        \caption{Building the cusp triangulation. In (A), the left frame shows three layers of $4$-punctured spheres, with a truncated puncture. Note the special edge $e$ on the truncated puncture, also shown in the right frame, which connects $S_{i-1}$ to $S_{i+1}$. Note that in the two figures on the right, the top and bottom vertices are identified, and in (B) we have rotated (vertically) by $\pi$ to make the picture more clear.}
        \label{fig:cusp_cross2}
\end{figure}

\begin{figure}
        \centering
        \begin{subfigure}[b]{0.2\textwidth}
                \includegraphics[scale=.9]{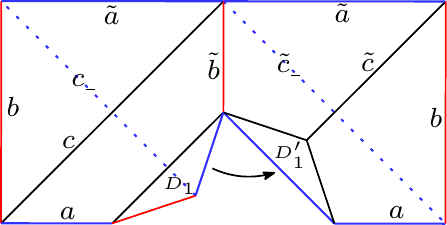}
                \caption{}
                \label{fig:cusp_trian1a}
        \end{subfigure}       
        \qquad\qquad
        \begin{subfigure}[b]{0.37\textwidth}
                \includegraphics[scale=.9]{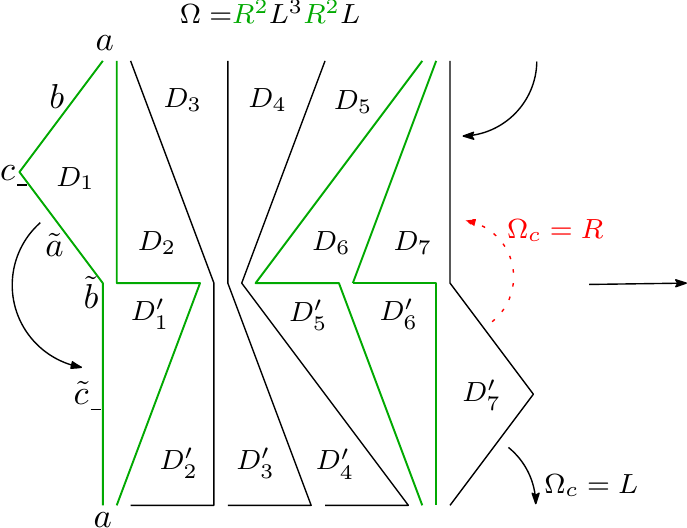}
                \caption{}
                \label{fig:cusp_trian1b}
        \end{subfigure}
       \qquad
               \begin{subfigure}[b]{0.25\textwidth}
                \includegraphics[scale=.9]{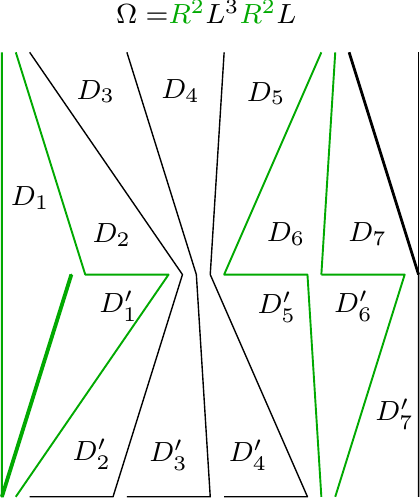}
                \caption{}
                \label{fig:cusp_trian1c}
        \end{subfigure}       
        \caption{The effect of clasping on the triangulation around a puncture. (A) shows $\Delta_1$, with $S_1$ below $S_2$, and edge colors of $S_1$ corresponding to colors in Figure \ref{fig:clasping}.}
        \label{fig:cusp_trian1}
\end{figure}

Before clasping, it is clear from the construction that the combinatorics around each of the four punctures is identical. Clasping identifies the punctures on $S_1=S\times\{0\}$ in pairs, and identifies the punctures on $S_c=S\times\{1\}$ in pairs, in an orientation preserving way. This means that for a $2$-component link, a cusp triangulation is obtained by gluing two puncture triangulations (as in Figure \ref{fig:cusp_trian1c}) along their front edges, and along their back edges, in an orientation preserving way. For a knot, the situation is similar, except that we glue all four puncture triangulations, always identifying front edges to front edges, and back to back, with orientation preserved. In both cases the lifted triangulation $\widetilde{T}$ of $\R^2$ is the same, except that the fundamental region for a knot is twice as large as for a link. Note that when $\Omega_1\neq \Omega_c$, the clasping triangle on the right is offset vertically from the clasping triangle on the left (as in Figure \ref{fig:cusp_trian1c}), whereas if $\Omega_1=\Omega_c$ this will not be the case. 

As a result of the above discussion, we can now give an algorithmic approach to constructing the lifted cusp triangulation $\widetilde{T}$ for an arbitrary word $\Omega=R^{\alpha_1}L^{\alpha_2}\dots L^{\alpha_n}$ (we will assume the last letter is $L$ for concreteness; the case where $\Omega_c=R$ is similar). This follows the approach of Sakuma-Weeks in \cite[II.4]{SaWe}, with some changes of notation. We start with a rectangle $D'=[0,1]\times [0,1] \subset \R^2$ divided into $c=\sum_i \alpha_i$ triangles, each corresponding to a letter of $\Omega$, as in Figure \ref{fig:T_alg1}. Vertices of $D'$ are labelled as shown, with $c_j=\sum_{i=1}^j \alpha_i$ for $1\le j\le n$, and $c_0=0$. To fill out $\R^2$ we first reflect $D'$ in its top edge to get its mirror $D$, so that $D\cup D'$ is a triangulation of a puncture (with triangles $D_i$ in $D$ and triangles $D_i'$ in $D'$), as in \ref{fig:cusp_trian1c}. We then rotate $D\cup D'$ by $\pi$ about $(0,1)$ (i.e., about the vertex labelled -1), and translate the resulting double of $D\cup D'$ vertically and horizontally to fill $\R^2$. Finally, we remove all edges $\overline{-1,1}$ and $\overline{r,c_n}$, where $r=c_{n-2}$ if $\alpha_n=1$, and $r=c_n-1$ otherwise (i.e. all images of the red edges in Figure \ref{fig:T_alg1}).

\begin{figure}[h]
    \centering
    \includegraphics[scale=.9]{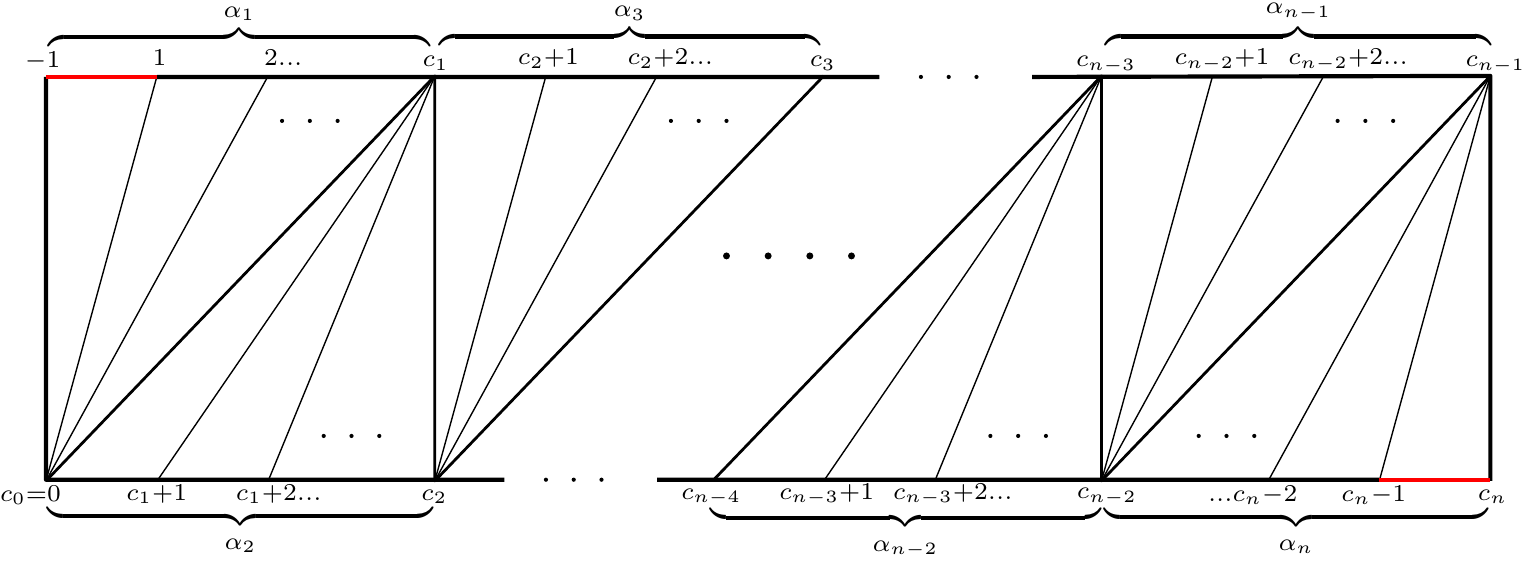}
    \caption{Triangulation of $D'=[0,1]\times [0,1]\subset \R^2$. The word $\Omega=R^{\alpha_1}L^{\alpha_2}\dots L^{\alpha_n}$ can be read from the triangulation. Here, $c_j=\sum_{i=1}^j \alpha_i$.}
    \label{fig:T_alg1}
\end{figure}

With this parametrization of the cusp triangulation in $\R^2$, deck transformations are generated by $(x,y)\mapsto (x,y+2)$ and $(x,y)\mapsto (x+k,y)$, where $k=2$ if $K=K(\Omega)$ has two components, and $k=4$ if it has one component. We observe that the long edge of each clasping triangle goes all the way around the meridian of the cusp, and these edges are unique in this respect. For this reason we call these edges \emph{meridional edges} (whether we are referring to them in $T$ or $\widetilde{T}$), and we call each connected component of their union in $\widetilde{T}$ a \emph{meridional line} (i.e., any line $x=c$, $c\in\Z$). A strip of adjacent non-clasping triangles that all meet the lines $y=m$ and $y=m+1$ (in an edge or vertex), for some $m\in \Z$, is called a \emph{horizontal strip} (see Figure \ref{fig:claspsandstrips}).

We will now describe a correspondence between edges and vertices of $\widetilde{T}$. Given an edge $e$ in $\widetilde{T}$, meaning a truncated tip of an ideal triangle in $\widetilde{\mathcal{T}}$, we have a corresponding edge in $\widetilde{\mathcal{T}}$: this is just the edge of $\widetilde{\mathcal{T}}$ across from $e$ in the ideal triangle, as in Figure \ref{fig:cusp_cross1}. Similarly, a vertex of $\widetilde{T}$ corresponds to the edge in $\widetilde{\mathcal{T}}$ that it is contained in. We say that an edge $e$ and a vertex $v$ of $\widetilde{T}$ \emph{correspond} if their corresponding edges in $\widetilde{\mathcal{T}}$ have the same slope (when viewed in $(\R^2\setminus \Z^2) \times I$).
Edge and vertex correspondence in $\widetilde{T}$, for edges and vertices that do not come from $S_1$ or $S_c$, can be read off of Figure \ref{fig:EV_cor2}, which shows the cusp cross section of a layer $\Delta_i$ with vertices and edges of the same slope labelled the same.
As for edges and vertices affected by clasping, we can easily read the correspondences off of the labellings in Figure \ref{fig:cusp_trian1} for the clasping of $S_1$, and the $S_c$ clasping works similarly.
This gives edge/vertex correspondences for $D\cup D'$, as shown in Figure \ref{fig:EV_cor} (as usual, we assume $\Omega_1=R$). A fundamental region of $T$ is constructed by gluing together either two or four copies of $D\cup D'$ by orientation reversing homeomorphisms $\{0\}\times [0,1] \to \{0\}\times [0,1]$ and $\{1\}\times [0,1] \to \{1\}\times [0,1]$, as previously discussed. Hence, the algorithmic construction of $\widetilde{T}$ by rotating $D\cup D'$ by $\pi$ about $(0,1)$ then translating to tile the plane respects edge valence, and so edge/vertex correspondence for all of $\widetilde{T}$ can be obtained in this way.  From here forward we will consider the edges of $\widetilde{T}$ to be labelled by the valence of a corresponding vertex, and we will refer to this number as the \emph{edge valence}.

\begin{figure}[h]
    \centering
    \includegraphics[scale=1]{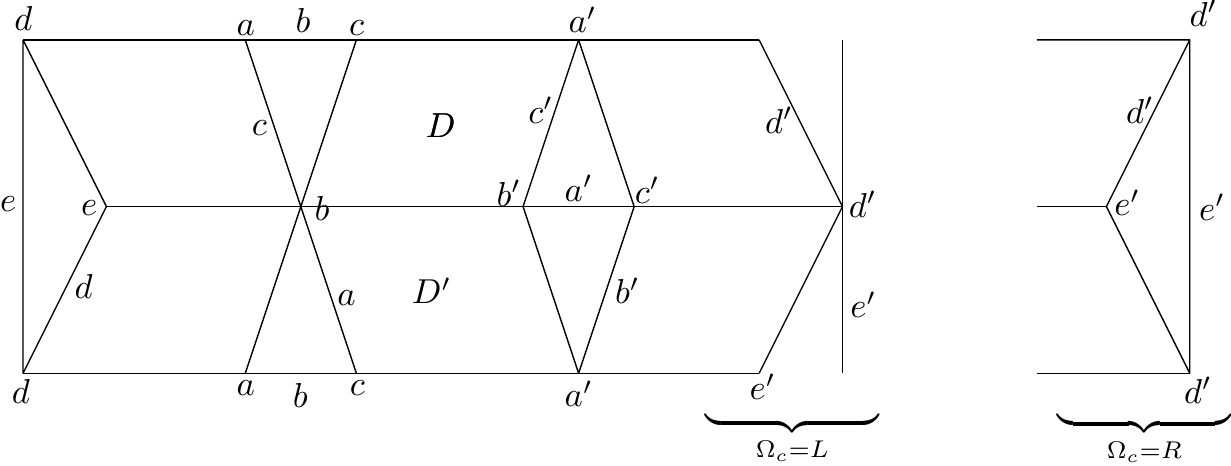}
    \caption{Edge/vertex correspondence in $\widetilde{T}$. Vertices and edges with the same slope (as edges in $\widetilde{\mathcal{T}}$), are labelled the same.}
    \label{fig:EV_cor}
\end{figure}

We summarize the preceding discussion in the following Lemma, part (d) of which corrects a minor error in the proof of Theorem II.3.1 in \cite{SaWe} (this error does not, however, affect the validity of their proof). Note that the relevant notation in \cite{SaWe} differs from ours in several ways: most importantly, what we call $val(i)$ they denote $d(i)$, and we follow a different indexing convention for vertices of $\tilde{T}$.

\begin{lemma}\label{sum_lem}
The lifted cusp triangulation $\widetilde{T}$ for the link given by a word $\Omega=R^{\alpha_1}L^{\alpha_2}R^{\alpha_3}\cdots L^{\alpha_n}$ has the following description:

\text{}\\[0 pt]
\hangindent=.5cm $\mathrm{\mathbf{(a)}}$ $\widetilde{T}$ is obtained from the triangulated rectangle $D'=[0,1]\times [0,1]\subset \R^2$, described by Figure \ref{fig:T_alg1}, as follows: reflect in $[0,1]\times \{1\}$ to get $D$, then rotate $D\cup D'$ about $(0,1)$, and translate the resulting two copies of $D\cup D'$ by $(x,y)\mapsto (x+2k,y+2m)$, $k,m\in \Z$, to tile $\R^2$.

\text{}\\[0 pt]
$\mathrm{\mathbf{(b)}}$ The deck group of $\widetilde{T}$ is generated by $(x,y)\mapsto (x,y+2)$ and $(x,y)\mapsto (x+\frac{4}{\epsilon},y)$, where $\epsilon\in \{1,2\}$ is the number of components of the link $K(\Omega)$.

\text{}\\[0 pt]
\hangindent=0.5cm$\mathrm{\mathbf{(c)}}$ Edge/vertex correspondence in $\widetilde{T}$ is as follows (see Figure \ref{fig:EV_cor}):

\text{}\\[0 pt]
\hangindent=1cm  $\bullet$ If $e$ is horizontal or $e$ is a meridional edge, then $e$ corresponds to the vertices across the two triangles adjacent to it.

\text{}\\[0 pt]
\hangindent=1cm  $\bullet$ If the lower endpoint of $e$ meets the line $y=k$, and the upper endpoint meets $y=k+1$, with k even (resp. odd), then $e$ corresponds to the vertex across the triangle to the left (resp. right) of $e$.

\text{}\\[0 pt]
\hangindent=0.5cm$\mathrm{\mathbf{(d)}}$ If $\Omega\notin \{R^2L^2, RL^m, RL^mR : m\ge 1\}$, then the vertices of $\widetilde{T}$, labelled as in Figure \ref{fig:T_alg1}, have valence as follows (recall that $r=c_{n-2}$ if $\alpha_n=1$, $r=c_n-1$ otherwise):

\text{}\\[0 pt]
\hangindent=1cm  $\bullet \quad val(c_i)=\left\{
\begin{array}{ll}
4\alpha_{i+1}+4 & \quad i\in \{0,n-1\}\\
2\alpha_{i+1}+4 & \quad 2\le i\le n-3 \quad or \quad i=1, \alpha_1>1 \quad or \quad i=n-2, \alpha_n>1\\
2\alpha_{i+1}+3 & \quad i=1, \alpha_1=1 \quad or \quad i=n-2, \alpha_n=1\\
\end{array}
\right.
$

\text{}\\[5 pt]
\hangindent=1cm  $\bullet\quad val(1)=\left\{
\begin{array}{ll}
3 & \alpha_1>1\\
2\alpha_2+3 & \alpha_1=1\\
\end{array}
\right.
$

\text{}\\[5 pt]
\hangindent=1cm
$\bullet\quad val(r)=\left\{
\begin{array}{ll}
3 & \alpha_n>1\\
2\alpha_{n-1}+3 & \alpha_n=1\\
\end{array}
\right.
$

\text{}\\[0 pt]
\hangindent=1cm  $\bullet \quad val(j) =4 \quad for \quad j\notin \{0,1,c_1, c_2,\dots, c_n, r\}$

\end{lemma}

\text{}\\[10 pt]
In particular, note that for all $\Omega \notin \{R^2L^2, RL^m, RL^mR : m\ge 1\}$, $val(j)$ is odd if and only if $j\in\{1,r\}$. This fact is key to showing that non-arithmetic $2$-bridge links cannot have hidden symmetries. Since a hidden symmetry restricts to an isometry of $\widetilde{\mathcal{T}}$, it is a \emph{simplicial automorphism} of $\widetilde{\mathcal{T}}$ (i.e., a homeomorphism $\widetilde{\mathcal{T}}\to\widetilde{\mathcal{T}}$ preserving the simplicial structure) and hence it is a simplicial automorphism of $\widetilde{T}$ that preserves edge valence. 

\begin{defn}
We denote by $\mathrm{Aut}_{ev}(\widetilde{T})$ the \emph{group of simplicial automorphisms of $\widetilde{T}$ that preserve edge valence}.  Note that if we identify $\tilde{T}$ with the horoball centered at $p$, then there is a natural injection $\mathrm{Stab}_{\mathrm{Aut}(\tilde{\mathcal{T}})}(p)\hookrightarrow \mathrm{Aut}_{ev}(\widetilde{T})$.
\end{defn}

By analyzing $\mathrm{Aut}_{ev}(\widetilde{T})$, which must preserve these odd valence vertices, we learn about the possible isometries of $\widetilde{\mathcal{T}}$. The first step in this process is the following lemma:

\begin{lemma}\label{SH_lemma} If $\Omega\notin\{RL, R^2L^2, RLR\}$, then $\mathrm{Aut}_{ev}(\widetilde{T})$ preserves clasping triangles and meridional edges. 
\end{lemma}

\begin{proof}
By the symmetry of the problem, we need only show that any triangle $\bigtriangleup_{1,0,0}$ with vertex labels $\{1,0,0\}$ maps to a clasping triangle. Let $f \in \mathrm{Aut}_{ev}(\widetilde{T})$, and let $\bigtriangleup_{a,b,b'}$ be the image of a triangle $\bigtriangleup_{1,0,0}$ under $f$, so that $1\mapsto a$. 

\textbf{Case 1:} $\Omega\notin \{R^kL^m, RL^mR^k\}$. Since $val(j)$ is odd if and only if $j\in \{1,r\}$, we must have $a\in \{1,r\}$. We will assume that $a=1$; the case $a=r$ is proved similarly. Then $b\in\{0,c_1,c_2,c_3\}$ since $val(0)=4\alpha_1+4\ge 8$ and all other vertices that could share an edge with $1$ have valence $4$.

If $val(1)=3$ (i.e., $\alpha_1>1$), then $b\in \{0,c_1\}$, since in this case no vertex $c_2$ or $c_3$ is connected to $1$   by an edge. If $b=c_1$, then we must have $\alpha_1=2$, so that $val(0)=4\alpha_1+4=12 =val(c_1)=2\alpha_2+4 \implies \alpha_2=4$, which means that $c_1+1$ must have valence $4$. But  $b=c_1$ also implies that $c_1+1$ is the image of the valence $3$ vertex of the clasping triangle that shares a meridional edge with $\bigtriangleup_{1,0,0}$, giving a contradiction. Thus $b=0$, and by the same argument we must also have $b'=0$.

If $val(1)\ne 3$, then $\alpha_1=1$ and $val(1)=val(c_1)=2\alpha_2+3$, and we must have $b\in \{0,c_2,c_3\}$. Also, $val(0)=4\alpha_1+4=8$.

If $b=c_2$, then $2\alpha_3+4=val(c_2)=val(0)=8$, so $\alpha_3=2$. This implies that $val(c_2+1)=4\ne 8$, so we must have $\bigtriangleup_{1,0,0}\mapsto \bigtriangleup_{1,c_2,0}$. This determines the image of the two non-clasping triangles $\bigtriangleup_{0,c_1,c_2}$ adjacent to $\bigtriangleup_{1,0,0}$, and we see that the $c_2$ vertex of one of these must be mapped to a $c_2+1$ vertex, which  is impossible since $val(c_2+1)=4\ne 8=val(c_2)$. 

If $b=c_3$ then $1=c_1$ and $c_3$ are connected by an edge, so $\alpha_3=1$, which forces the other $0$-labelled vertex of $\bigtriangleup_{1,0,0}$ to map to $c_2$, which is impossible by the above argument. Hence $b=0$, and by the same argument we have $b'=0$.

That meridional edges map to meridional edges is immediate since $\Omega\notin \{R^kL^m, RL^mR^k\}$ implies that clasping triangles have a unique odd valence vertex, i.e., the vertex not meeting a meridional edge.

\textbf{Case 2:} $\Omega=R^kL^m$ and $\Omega \notin\{RL, R^2L^2\}$. If $k=1$, then
clasping triangles either have vertices with valences $8,8,4m+2$ or $3,4m+2,4m+2$, and they are the only triangles in $\widetilde{T}$ with such a triple of valences. If $k\ne 1$ then clasping triangles either have vertices with valences $3,4k+4,4k+4$ or $3,4m+4,4m+4$, and they are the only triangles in $\widetilde{T}$ with such a triple of valences. Furthermore, in every case two of these vertices have equal valence and the third has distinct valence, so meridional edges must be preserved.

\textbf{Case 3:} $\Omega=RL^mR^k$, $\Omega\neq RLR$. Then $\alpha_1=1\implies val(0)=8$. If $k>1$ then $val(1)\neq val(r)$, so $1\mapsto 1$ and we must have vertices labelled $0$ mapping to vertices labelled $0$ or $c_2=c_{n-1}$. But $val(c_{n-1})=4k+4\ne 8$, so $0\mapsto 0$. If $k=1$ then clasping triangles all have vertices with valences $8,8,2m+2$, and they are the only triangles in $\widetilde{T}$ with this triple of valences. Furthermore, meridional edges are preserved since even when $m=3$ (so that $2m+2=8$), the vertices labelled $1=r$ are combinatorially distinct from the vertices labelled $0$ and $c_{n-1}$: vertices labelled $1$ have four edges connecting them to valence $4$ vertices, while vertices labelled $0$ and $c_{n-1}$ have only two such edges. 
\end{proof}

\begin{cor}
If $\Omega\notin\{RL, R^2L^2, RLR\}$, then $\mathrm{Aut}_{ev}(\widetilde{T})$ preserves horizontal strips of $\widetilde{T}$.
\end{cor}
\begin{proof}
Let $C$ be the clasping triangle in the first quadrant of $\R^2$ with a vertex at the origin. $C$ is adjacent to two horizontal strips; let $H$ be the one adjacent to the $x$-axis, and let $C'$ be the other clasping triangle adjacent to $H$. Let $\gamma$ be the path directly across $H$ connecting the midpoints of the edges of adjacency with $C$ and $C'$. Consider the image of $\gamma$ under a simplicial automorphism $f:\widetilde{T}\to \widetilde{T}$. Since $\gamma$ crosses exactly $c-2$ triangles, so must $f(\gamma)$. By Lemma \ref{SH_lemma}, $f$ maps $e$ and $e'$ to edges of clasping triangles, which are adjacent to distinct meridional lines since $C$ and $C'$ are, and $f$ maps triangles crossed by $\gamma$ to non-clasping triangles, so $\gamma$ must be mapped into some number of vertically stacked horizontal strips. Since $\gamma$ crosses all triangles transversely, if $f(\gamma)$ jumps from one horizontal strip to another the number of triangles it crosses must be one more that if it did not make the jump, as shown in Figure \ref{fig:HStoHS}. Hence $f(\gamma)$ must be contained in one horizontal strip, the image of $H$.
\end{proof}

\begin{figure}[h]
    \centering
    \includegraphics[scale=.9]{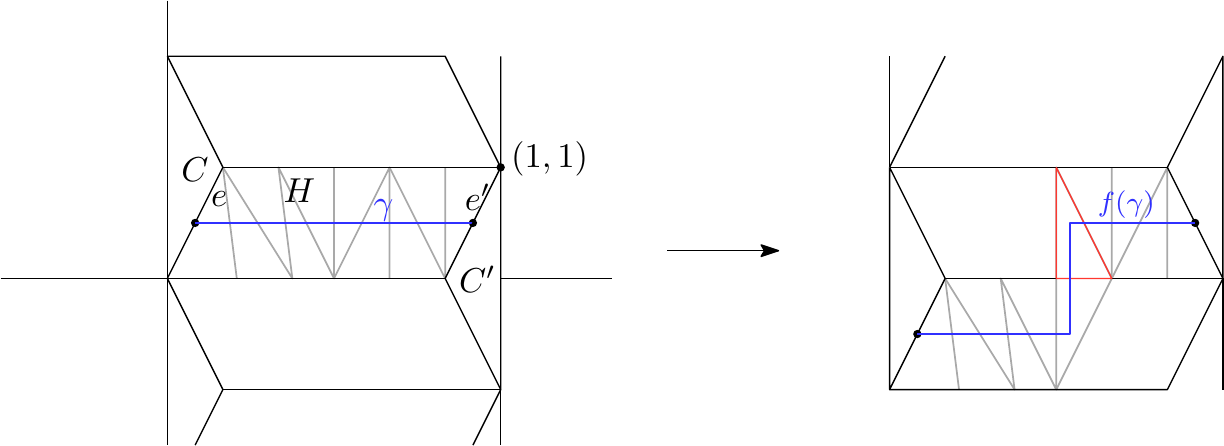}
    \caption{If $H$ maps into more than one horizontal strip, then $f(\gamma)$ traverses more than $c-2$ triangles, which is impossible.}
    \label{fig:HStoHS}
\end{figure}

Recall that in our algorithmic construction of $\widetilde{T}$, we chose coordinates so that the rectangle $D'$ shown in Figure \ref{fig:T_alg1} is identified with $[0,1]\times [0,1]\subset \R^2$. We have the following theorem:

\begin{thm}\label{SH_implies_iso}
If $\Omega\notin\{RL, R^2L^2, RLR\}$, then $\mathrm{Aut}_{ev}(\widetilde{T})$ is generated by the deck transformations and a subset of the  following:
\begin{itemize}
\item \textbf{Orientation-Preserving}: the rotations $\rho_1$, $\rho_2$, and $\rho_3$ about $(1,1)$, $(2,1)$, and $(\frac{1}{2},1)$, respectively, by an angle $\pi$.
\item \textbf{Orientation Reversing}: the glide reflection $g$ given by the reflection across $x=\frac{1}{2}$ composed with $(x,y)\mapsto (x,y+1)$ .
\end{itemize} 
Furthermore, we always have $\rho_1,\rho_2\in \mathrm{Aut}_{ev}^+(\widetilde{T})$, and $\rho_3\in \mathrm{Aut}_{ev}^+(\widetilde{T})$ (resp. $g\in \mathrm{Aut}_{ev}(\widetilde{T})$) if and only if $\rho_3$ (resp. $g$) is a simplicial automorphism.
\end{thm}

\begin{proof}
Let $f\in \mathrm{Aut}_{ev}(\widetilde{T})$, and let $E$ be the union of all edges of horizontal strips and clasping triangles, as shown in Figure \ref{fig:claspsandstrips}. Since $f$ maps clasping triangles to clasping triangles, and horizontal strips to horizontal strips, it must map $E$ to itself. Since the simplicial structure of the triangulation within each horizontal strip must be preserved, and since we may assume all clasping triangles are congruent and triangles within each strip are uniformly sized, $f$ is forced to be a Euclidean isometry of $\R^2$. Let $\rho_4$ be the rotation by $\pi$ about the point $(\frac{1}{2},\frac{1}{2})$, and let $r_y$ be the reflection about the line $y=1$. We first consider the possible Euclidean isometries preserving $E$:

\hangindent=.5cm $\bullet$\,\,\textbf{translations}: translations must preserve the integer lattice, so modulo deck transformations they have the form $\tau_{i,j}:(x,y)\mapsto (x+i,y+j)$, $i\in\{0,1,2,3\}$, $j\in \{0,1\}$. Since $\tau_{0,1}$, and $\tau_{2,1}$ do not preserve $E$, and $\tau_{0,0}$ is trivial, we are left with
\[
\tau_{1,0}=\rho_1\circ \rho_3; \quad \tau_{2,0}=\rho_2\circ \rho_1; \quad \tau_{3,0}=\rho_2\circ \rho_3; \quad \tau_{1,1}=\rho_1\circ\rho_4; \quad \tau_{3,1}=\rho_2\circ\rho_4
\]
and their inverses.

$\bullet$\,\,\textbf{rotations}: since meridional lines and integer lattice points must be preserved, any rotation must be by an angle $\pi$ about a point $(\frac{k}{2},\frac{m}{2})$, $k,m\in \Z$. The rotations about $(1,\frac{1}{2})$ and $(2,\frac{1}{2})$ do not preserve clasping triangles, so modulo deck transformations we are left with $\rho_1, \rho_2, \rho_3, \rho_4$, and the rotations
\[
\rho_4\circ \rho_2\circ \rho_1; \quad \rho_3\circ \rho_2\circ \rho_1
\]
about $(\frac{3}{2},\frac{1}{2})$ and $(\frac{3}{2},1)$, respectively.

$\bullet$\,\,\textbf{reflections}: reflections must preserve meridional lines and clasping triangles, so possible lines of reflection are $x=\frac{k}{2}$ or $y=k$, $k\in \Z$. Modulo deck transformations, we get the reflection $r_y$ across $y=1$, and the reflections $r_i$ across the lines $x=i$, $i\in \{\frac{1}{2}, 1, \frac{3}{2},2\}$. We have
\[
r_1=r_y\circ \rho_1; \quad r_2=r_y\circ \rho_2; \quad r_{\frac{1}{2}}=r_y\circ \rho_3; \quad r_{\frac{3}{2}}=r_{\frac{1}{2}}\circ \rho_2\circ \rho_1
\]

$\bullet$\,\,\textbf{glide reflections}: Since simplicial automorphisms preserve merdional lines and clasping triangles, the reflection component of the glide reflection must be across a line $x=\frac{k}{2}$ or $y=k$, $k\in \Z$. If the reflection is across $x=k \in \Z$, then the translation must be $(x,y)\mapsto (x,y+2n)$, $n\in \Z$, so modulo deck transformations this is a pure reflection, and can be ruled out. Thus we are left with the glide reflection $g=\tau_{0,1}\circ r_{\frac{1}{2}}$, given by the reflection across $x=\frac{1}{2}$ followed by the translation $(x,y)\mapsto (x,y+1)$, and the compositions
\[
r_y\circ \tau_{1,0}=r_y\circ \rho_1\circ\rho_3; \quad r_y\circ \tau_{2,0}=r_y\circ \rho_2\circ \rho_1; \quad r_{\frac{3}{2}}\circ \tau_{0,1}=g\circ \rho_2\circ \rho_1,
\]
all others being obtained by composing with deck transformations. 

\text{}\\[0 pt]
\hangindent=0cm
We show that $r_y\notin \mathrm{Aut}_{ev}(\widetilde{T})$ by considering edge valences near a clasping triangle. Using the edge/vertex correspondences from Figure \ref{fig:EV_cor}, we obtain the four pictures in Figure \ref{fig:or_hs}, which correspond to the cases $\alpha_1\ge 3$, $\alpha_1=2$, $\alpha_1=1\neq \alpha_2$, and $\alpha_1=1=\alpha_2$, respectively (note that $\Omega$ non-arithmetic implies $\Omega\notin \{ RL, RLR, R^2L^2\}$). For the first three pictures it is clear that $r_y$ does not preserve edge valence. For the last picture, if $r_y\in \mathrm{Aut}_{ev}(\widetilde{T})$ then $c=d=8$, so that $\alpha_3=2$, which implies $8=d=a=4$, a contradiction. Hence $r_y\notin \mathrm{Aut}_{ev}(\widetilde{T})$. 

In order to rule out $\rho_4$ and the compositions above involving $\rho_4$ and $r_y$, we will first need to establish the last assertion of the theorem, namely that we always have $\rho_1,\rho_2\in \mathrm{Aut}_{ev}(\widetilde{T})$, and  $\rho_3$ and $g$ are in $\mathrm{Aut}_{ev}(\widetilde{T})$ if and only if they are simplicial automorphisms of $\widetilde{T}$. To see this first note that $\rho_1$ and $\rho_2$ are always simplicial automorphisms (by construction of $\widetilde{T}$). Thus we need only show that if any of $g$, $\rho_1$, $\rho_2$, or $\rho_3$ is a simplicial homeomorphism, then it is in $\mathrm{Aut}_{ev}(\widetilde{T})$. But this follows from the fact that each of $g$, $\rho_1$, $\rho_2$, and $\rho_3$ preserve the edge/vertex correspondence given in Lemma \ref{sum_lem}(c) (shown graphically in Figure \ref{fig:EV_cor}). In particular, each of these maps switches the parity of $k$ in part (c) of the Lemma, but also exchanges right and left. Thus, if $g$ is simplicial, it preserves vertex valence, and since it also preserves edge/vertex correspondence, it must preserve edge valence, i.e., $g\in\mathrm{Aut}_{ev}(\widetilde{T})$. The same holds for $\rho_1,\rho_2$, and $\rho_3$, so the assertion is proved.

Now, suppose that $\rho_4\in \mathrm{Aut}_{ev}(\widetilde{T})$. First, observe that $g=\tau_{0,1}\circ r_{\frac{1}{2}}=r_{\frac{1}{2}}\circ\tau_{0,1}= (\rho_3\circ r_y)\circ (\rho_3\circ \rho_4)=(\rho_3\circ r_y \circ \rho_3)\circ \rho_4=r_y\circ \rho_4$. Since $r_y$ is always a simplicial automorphism (by construction of $\widetilde{T}$), $\rho_4\in \mathrm{Aut}_{ev}(\widetilde{T})$ implies that $g$ is a simplicial automorphism. 
Thus by the above paragraph, $g\in \mathrm{Aut}_{ev}(\widetilde{T})$. But $g, \rho_4 \in \mathrm{Aut}_{ev}(\widetilde{T})$ implies that $r_y\in \mathrm{Aut}_{ev}(\widetilde{T})$, a contradiction.

Thus we can rule out the compositions $\tau_{1,1}, \tau_{3,1}, \rho_4\circ \rho_2\circ \rho_1, r_1, r_2$, and $r_y\circ \tau_{2,0}$. For $r_\frac{1}{2}$ and $r_y\circ \tau_{1,0}$, since $r_y$ is always a simplicial homeomorphism, the composition is simplicial if and only if $\rho_3$ is. But then by the above observation it follows that $\rho_3$ preserves edge valence, so the composition cannot preserve edge valence (becuase $r_y$ does not). Last, $r_{\frac{3}{2}}$ can now be ruled out since $r_{\frac{1}{2}}\notin \mathrm{Aut}_{ev}(\widetilde{T})$.

Since the only compositions we have not ruled out are generated by $\rho_1, \rho_2, \rho_3$, and $g$, and since compositions involving $\rho_3$ (resp. $g$) are in $\mathrm{Aut}_{ev}(\widetilde{T})$ if and only if $\rho_3$ (resp. $g$) is, the result follows.
\end{proof}

\begin{figure}
        \centering
        \begin{subfigure}[b]{0.28\textwidth}
                \includegraphics[scale=.9]{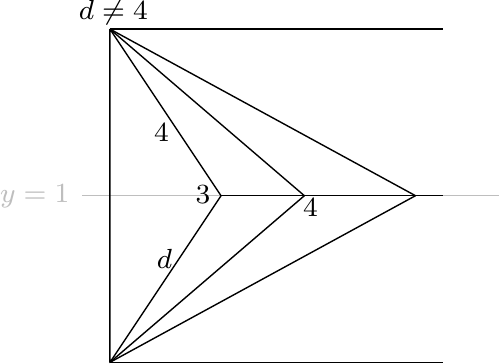}
                \caption{}
                \label{fig:or_hs_1}
        \end{subfigure}
        \begin{subfigure}[b]{0.22\textwidth}
                \includegraphics[scale=.9]{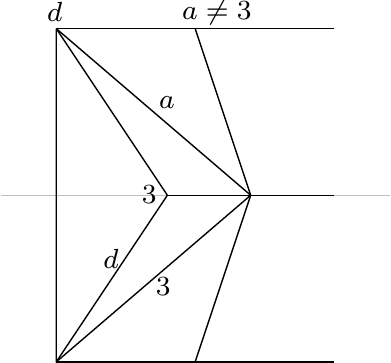}
                \caption{}
                \label{fig:or_hs_2}
        \end{subfigure}
         \begin{subfigure}[b]{0.22\textwidth}
                \includegraphics[scale=.9]{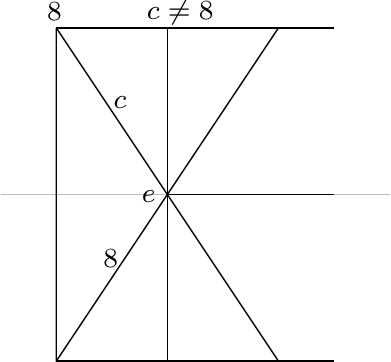}
                \caption{}
                \label{fig:or_hs_3}
        \end{subfigure}
         \begin{subfigure}[b]{0.22\textwidth}
                \includegraphics[scale=.9]{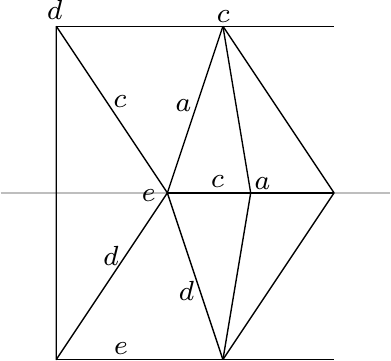}
                \caption{}
                \label{fig:or_hs_4}
        \end{subfigure}
        \caption{Reflecting by $r_y$ about $y=1$, for the cases $\alpha_1\ge 3$, $\alpha_1=2$, $\alpha_1=1\ne \alpha_2$, and $\alpha_1=1=\alpha_2$, from left to right respectively.}
        \label{fig:or_hs}
\end{figure}

\begin{remark} In Corollary \ref{SH_implies_iso} we have described a set \emph{containing} the generators of  $\mathrm{Aut}_{ev}(\widetilde{T})$, but we do not know whether they are all in fact generators. We will easily obtain in Section \ref{sec:hiddensym} a complete description of this group.
\end{remark}

\begin{figure}[h]
    \centering
    \includegraphics[scale=.8]{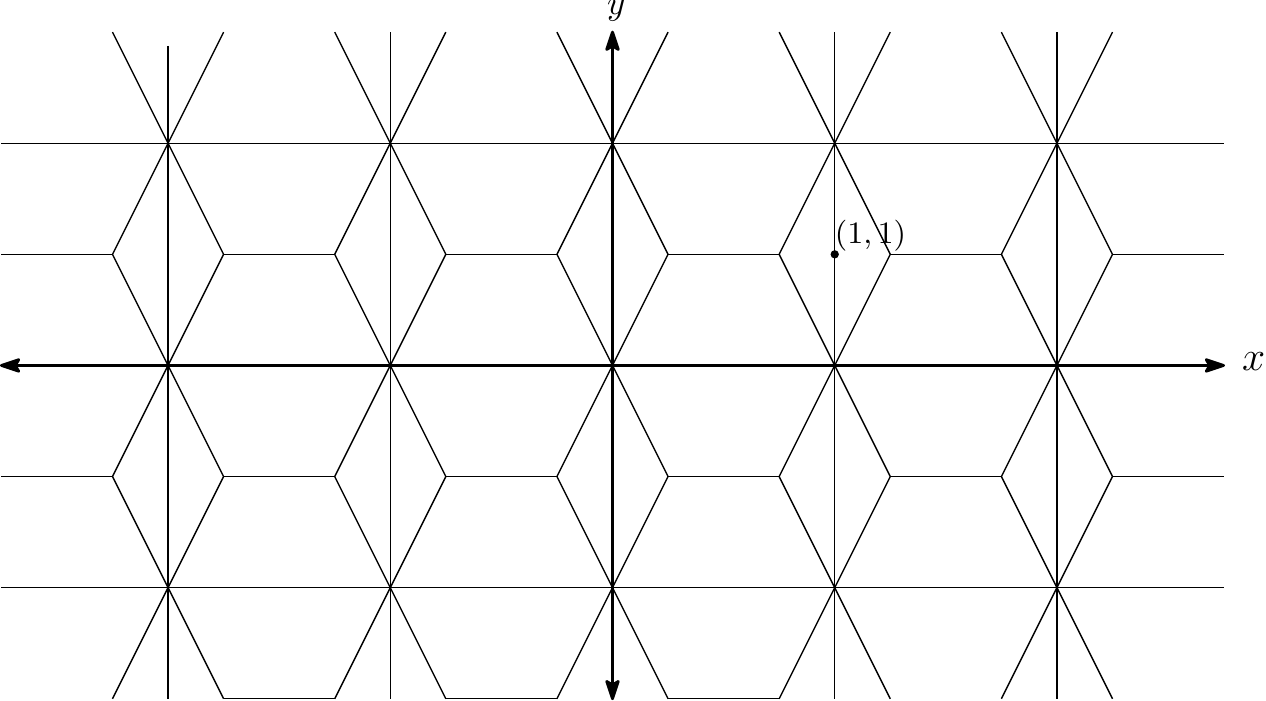}
    \caption{The union $E$ of all edges of horizontal strips and clasping triangles, in the case where $\Omega_c=R$. If $\Omega_c=L$ then the clasping triangles adjacent to line $y=k$, $k$ odd, will be shifted vertically by $1$, and horizontal strips will be parallelagrams.}
    \label{fig:claspsandstrips}
\end{figure}


\section{Symmetries of $2$-bridge link complements}
\label{sec:sym}

Let $M = \mathbb{S}^{3} \setminus K(\Omega)$, and let $Sym(M)$ denote the symmetries of $M$, i.e., $Sym(M)$ is the group of self-homeomorphisms of $M$ up to isotopy. Here, we describe the action of $Sym(M)$ on the triangulation $\widetilde{T}$.  First, Theorem \ref{thm:Symclass} gives a classification of the symmetries of $M$ in terms of the word $\Omega$. This theorem comes from combining Theorem II.3.2 and Lemma II.3.3 in \cite{SaWe} and translating from $\left[a_{1}, a_{2}, \dots, a_{n}\right]$ to the word $\Omega$ given by the following dictionary: $a_{1} = \alpha_{1}+1$, $a_{i} = \alpha_{i}$ for $i \leq 2 \leq n-1$, and $a_{n} = \alpha_{n}+1$. In \cite{SaWe}, these symmetries are called automorphisms of the triangulation $\mathcal{T}$ of $M$ described in Section \ref{sec:CuspT}. Since by \cite{Gu} this triangulation is now known to coincide with the canonical triangulation of $M$, we know these automorphisms actually correspond to all of the symmetries of $M$. 

We let $Sym^{+}(M)$ denote the subgroup of $Sym(M)$ consisting of orientation-preserving symmetries. We say that $\Omega$ is \textit{palindromic} if $\alpha_{i} = \alpha_{n-i+1}$ for all $1 \leq i \leq n$. 

\begin{thm}[\cite{SaWe},\cite{Gu}]
\label{thm:Symclass}
Let $M = \mathbb{S}^{3} \setminus K(\Omega)$ be any hyperbolic $2$-bridge link complement. Then $Sym(M) = Sym^{+}(M) \cong Z_{2} \oplus Z_{2}$ if and only if $\Omega$ is not palindromic. When $\Omega$ is palindromic, then we have the following possibilities:
\begin{itemize}
\item If $n$ is even, then $Sym(M) \cong D_{4}$ and $Sym^{+}(M)  \cong Z_{2} \oplus Z_{2}$.
\item If $n$ is odd and $\alpha_{\frac{n+1}{2}}$ is odd, then $Sym(M) = Sym^{+}(M) \cong D_{4}$.
\item If $n$ is odd and $\alpha_{\frac{n+1}{2}}$ is even, then $Sym(M) = Sym^{+}(M)  \cong Z_{2} \oplus Z_{2} \oplus Z_{2}$.
\end{itemize}
\end{thm}

Note that the $2$-bridge link complements with orientation-reversing symmetries are exactly those with $n$ even and $\Omega$ palindromic.

We would like to understand how these symmetries act on $\widetilde{T}$. In order to accomplish this, we will first show that  $Sym(M) = Sym(\mathbb{S}^{3}, K(\Omega))$. Here, $Sym(\mathbb{S}^{3}, K(\Omega))$ denotes the symmetries of $(\mathbb{S}^{3}, K(\Omega))$, that is, the group of self-homeomorphisms of the pair $(\mathbb{S}^{3}, K(\Omega))$ up to isotopy. Mostow--Prasad rigidity implies that $Sym(M) \supseteq Sym(\mathbb{S}^{3}, K)$ for any hyperbolic link $K$. In fact, if $K$ is a hyperbolic knot, then $Sym(M) = Sym(\mathbb{S}^{3}, K)$ by the Knot Complement Theorem \cite{GL}.  However, here we do not rely on the Knot Complement Theorem, and in addition, we prove the desired equality for both hyperbolic $2$-bridge knots and hyperbolic $2$-bridge links with two components. Once we have established this correspondence, we can determine how these symmetries act on the cusp triangulation, $T$. From here, we just lift this action of $Sym(M)$ on $T$ to the universal cover $\mathbb{R}^{2}$, to get the corresponding action on $\widetilde{T}$. 

The following proposition is certainly known by the experts in the field. However, the authors were unable to find a reference in the literature. 

\begin{prop}
	\label{prop:sym}
	Let $M = \mathbb{S}^{3} \setminus K(\Omega)$ be a hyperbolic $2$-bridge link complement. Then $Sym(M) = Sym(\mathbb{S}^{3}, K(\Omega))$.
\end{prop}

\begin{proof}
The work of Gu\'{e}ritaud \cite{Gu} shows that $\mathcal{T}$ is in fact the canonical triangulation of any such hyperbolic $2$-bridge link complement $M$. Thus, $\mathrm{Aut}(\mathcal{T})$, the group of combinatorial automorphisms of this triangulation, is isomorphic to $Sym(M)$. The description of $\mathrm{Aut}(\mathcal{T})$ given in \cite[pp. 415-416]{SaWe} implies that it preserves the meridian(s) of $K(\Omega)$, and hence extends to an action on $(\mathbb{S}^{3}, K(\Omega))$. As a result, the natural inclusion from $Sym(\mathbb{S}^{3}, K(\Omega))$ into $Sym(M)$ is surjective, giving the desired isomorphism. 
\end{proof}

Since $Sym(M)$ is isomorphic to $Sym(\mathbb{S}^{3}, K(\Omega))$, we will no longer distinguish between symmetries of a hyperbolic $2$-bridge link and its complement. Below, we provide visualizations of these symmetries, which will be useful in the proofs of Lemma \ref{lem:maxminedges} and Proposition \ref{syms}. For more visualizations of $2$-bridge link symmetries, see \cite{BlMo}, \cite{BoSi}, and \cite{Sa}.

\begin{figure}
	\includegraphics[scale=.9]{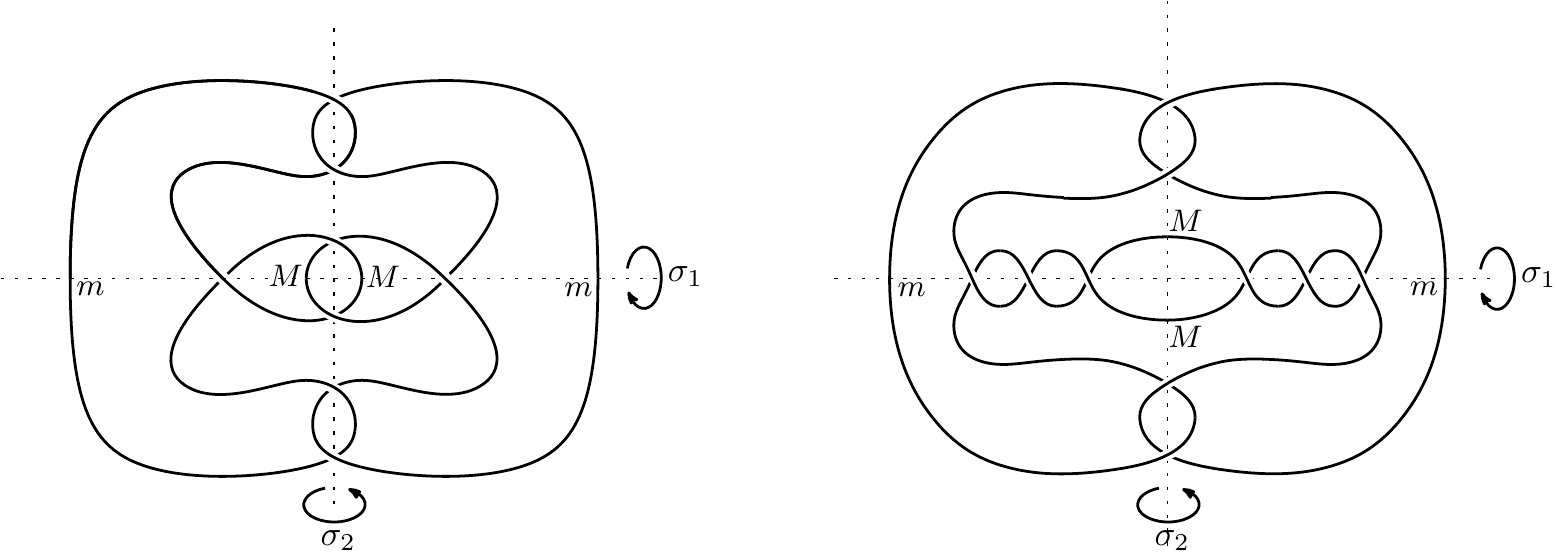}
	\caption{Tri-symmetric projections of a $2$-bridge link with two components (left) and a $2$-bridge knot (right). The axes of symmetry for $\sigma_{1}$ and $\sigma_{2}$ are given in both projections. Maxima are labeled M and minima are labeled m.}
	\label{fig:trysim}
\end{figure}

\begin{figure}
\centering
\begin{subfigure}[b]{0.40\textwidth}
	\includegraphics[scale=1]{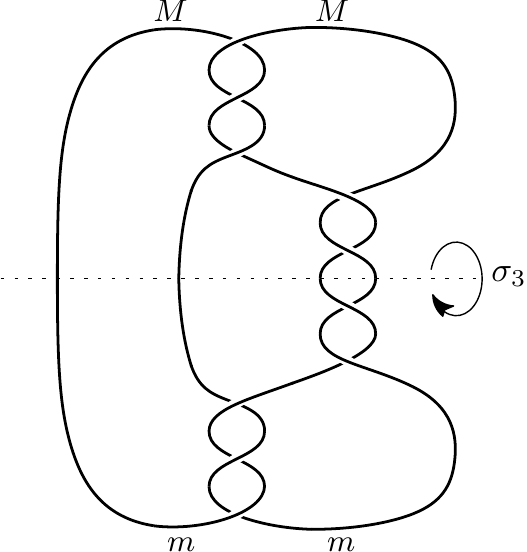}
	\caption{}
	\label{fig:palindromicsim2}
\end{subfigure}
\qquad
\begin{subfigure}[b]{0.45\textwidth}
	\includegraphics[scale=.56]{even_pal_sym}
	\caption{}
	\label{fig:even_pal_sym}
\end{subfigure}
\caption{To the left is the standard projection of $K(\Omega)$ with $\Omega$ palindromic and $n$ odd. To the right is a depiction of $K(\Omega)$ in $\mathbb{R}^{3}$ (with knot strands connecting at infinity) with $\Omega$ palindromic and $n$ even. Both visuals show a symmetry $\sigma_{3}$ of $K(\Omega)$. Maxima are labeled M and minima are labeled m.}
\label{palindromicsym}
\end{figure}

Recall that any $2$-bridge link $K(\Omega)$ can be isotoped so that its projection has exactly two maxima and two minima. In all four link diagrams given in Figure \ref{fig:trysim} and Figure \ref{palindromicsym} the corresponding maxima and minima are labeled. In what follows, we will examine how $Sym(M)$ acts on these maxima and minima, and ``meridional edges'' of $\mathrm{Aut}(\mathcal{T})$ that wrap around them.  For an arbitrary link $L \subset \mathbb{S}^{3}$, this would be an issue since the maxima and minima don't have to be preserved up to isotopy, and $Sym(\mathbb{S}^{3}, L)$ is a group of homeomorphisms up to isotopy. However, for a $2$-bridge link, from the work of Schubert \cite{Schu} we know that the set of maxima and minima will be preserved up to isotopy, and so, we are justified in using different projections of $K(\Omega)$ to analyze how symmetries act on the maxima and minima.  

\begin{lemma}
\label{lem:maxminedges}
Each ``meridional edge'' of $T$ wraps around a maximum or minimum of $K(\Omega)$. These meridional edges alternate between ones that wrap around maxima and minima.
\end{lemma}
\begin{proof}
In all cases, $\mathcal{T}$, the canonical triangulation of $\mathbb{S}^{3} \setminus K(\Omega)$, has exactly four meridional edges, and $K(\Omega)$ has exactly four extrema. These meridional edges of $\mathcal{T}$ result from clasping. See Section \ref{sec:CuspT} for details on how clasping the innermost and outermost $4$-punctured spheres, $S_{c}$ and $S_{1}$, affects $\mathcal{T}$. Specifically, clasping $S_{1}$ introduces two meridional edges, each one going around one of the strands of the outermost crossing of $K(\Omega)$. We get the other two meridional edges from clasping $S_{c}$, each one going around one of the strands of the innermost crossing. See Figure \ref{fig:clasping} for how clasping forms these meridional edges.  The two meridional edges coming from clasping $S_{1}$ each go around a maximum of $K(\Omega)$, while the two meridional edges coming from clasping $S_{c}$ each go around a minimum of $K(\Omega)$. Since there are exactly four meridional edges in $\mathcal{T}$ and exactly four meridional edges in $T$, these sets must correspond with one another. Thus, each meridional edge of $T$ wraps around a maximum or minimum of $K(\Omega)$. These meridional edges alternate between wrapping around maxima and minima since if we orient $K(\Omega)$, our path alternates between traversing maxima and minima.
\end{proof}

We now consider the lifts of the meridional edges of $T$ to $\widetilde{T}$. In what follows, we shall call the lifts of meridional edges of $T$ that wrap around a maximum of $K(\Omega)$ \textit{maximal meridional edges.} Similarly, we shall call the lifts of the meridional edges of $T$ that wrap around a minimum of $K(\Omega)$ \textit{minimal meridional edges.}

We now describe how the symmetries of a hyperbolic $2$-bridge link complement act on $\widetilde{T}$. Recall that $n$ is the number of syllables in the word $\Omega$. If $K(\Omega)$ is a two component link, then we say $\widetilde{T} = \widetilde{T}_{1} \cup \widetilde{T}_{2}$, where $\widetilde{T}_{1}$ and $\widetilde{T}_{2}$ are identical triangulations of $\mathbb{R}^{2}$, coming from lifting an equal volume cusp cross-section of $\mathbb{S}^{3} \setminus K(\Omega)$.

Recall that $\sigma_{1}$, $\sigma_{2}$, and $\sigma_{3}$ are the symmetries of $Sym(\mathbb{S}^{3}, K(\Omega))$ desribed above and shown in Figure \ref{fig:trysim} and Figure \ref{palindromicsym}.

\begin{prop}\label{syms}
$Sym(M) = Sym(\mathbb{S}^{3}, K(\Omega))$ acts on $\widetilde{T}$ (up to deck transformations) in the following manner:

If $K(\Omega)$ is a knot, then
\begin{itemize}
\item $\sigma_{1}$ acts as a rotation of $\pi$ about $(1,1)$, and 
\item $\sigma_{2}$ acts as a rotation of $\pi$ about $(2,1)$.
\end{itemize}
If $K(\Omega)$ is a two component link, then
\begin{itemize}
\item $\sigma_{1}$ acts as a rotation of $\pi$ about $(1,1)$ in both $\widetilde{T}_{1}$ and $\widetilde{T}_{2}$, and 
\item $\sigma_{2}$ exchanges $(\mathbb{R}^{2}, \widetilde{T}_{1})$ and $(\mathbb{R}^{2}, \widetilde{T}_{2})$ by the identity map.
\end{itemize}
If $\Omega$ is palindromic, then
\begin{itemize}
\item if $n$ is odd, $\sigma_{3}$ acts as a rotation of $\pi$ about $(\frac{1}{2}, 1)$, and
\item if $n$ is even, $\sigma_{3}$ acts as a glide reflection where we reflect across the line $x = \frac{1}{2}$ and translate by $(x,y) \rightarrow (x, y+1)$ (possibly composed with the rotations $\sigma_{1}$ and $\sigma_{2}$).
\end{itemize}
\end{prop}

\begin{proof}
First, we claim that any symmetry of $M$ acts on $(\mathbb{R}^{2}, \widetilde{T})$ by an isometry of $\mathbb{R}^{2}$. A priori, a symmetry of $M$ gives rise only to an element $f$ of $\mathrm{Aut}_{ev}(\widetilde{T})$ since this triangulation is metrically distorted in our construction. By Theorem \ref{SH_implies_iso}, any such simplicial homeomorphism (that preserves edge valences) of $\widetilde{T}$ is a composition of deck transformations (which are specific translations) and a specific set of rotations, reflections, and glide reflections. Thus, any such $f$ must be a Euclidean isometry.     

First, we consider the symmetries $\sigma_{1}$ and $\sigma_{2}$ of $M$ that generate a subgroup of $Sym(M)$ isomorphic to $Z_{2} \oplus Z_{2}$. By Theorem \ref{thm:Symclass}, these symmetries are always orientation-preserving, and so, we just need to consider rotations and translations of $\mathbb{R}^{2}$. We do this in two cases.

\textbf{Case 1: $K(\Omega)$ is a knot.}
In this case, we note the following properties of $\sigma_{1}$ and $\sigma_{2}$. These properties come from examining the tri-symmetric projection given in Figure \ref{fig:trysim}. 
\begin{itemize}
\item $\sigma_{1}$ exchanges the maxima of $K(\Omega)$ while fixing the minima of $K(\Omega)$.   
\item $\sigma_{2}$ exchanges the minima of $K(\Omega)$ while fixing the maxima of $K(\Omega)$.
\item $\sigma_{1}$ and $\sigma_{2}$ change the orientation of the longitude of $K(\Omega)$.
\end{itemize}
Since both $\sigma_{1}$ and $\sigma_{2}$ change the orientation of the longitude, they cannot be translations, and so, must be rotations. By Lemma \ref{lem:maxminedges}, $\sigma_{1}$ must exchange the maximal meridional edges while fixing the two minimal meridional edges. Thus, up to deck transformations, $\sigma_{1}$ must be a rotation of $\pi$ about $(1,1)$. Similarly, up to deck transformations, $\sigma_{2}$ must be a rotation of $\pi$ about $(2,1)$.

\textbf{Case 2: $K(\Omega)$ is a 2-component link.}
Here, we once again note several important features of $\sigma_{1}$ and $\sigma_{2}$ acting on $(\mathbb{S}^{3}, K(\Omega))$ which come from examining the tri-symmetric projection in Figure \ref{fig:trysim}.
\begin{itemize}
\item $\sigma_{1}$ sends each component of $K(\Omega)$ to itself, with maxima mapping to maxima and minima mapping to minima.
\item $\sigma_{2}$ exchanges the two link components, with maxima mapping to maxima and minima mapping to minima.
\item $\sigma_{1}$ changes the orientations of both of the longitudes of $K(\Omega)$, while $\sigma_{2}$ preserves these orientations.  
\end{itemize}
Since $\sigma_{1}$ is an orientation-preserving symmetry that switches the orientation of both of the longitudes, it must act as a rotation on both copies of $\mathbb{R}^{2}$. Up to deck transformations, the only possible rotation that maps the two maximal meridional edges to themselves and maps the two minimal meridional edges to themselves is a rotation of $\pi$ about $(1,1)$ in both $(\mathbb{R}^{2}, \widetilde{T}_{1})$ and $(\mathbb{R}^{2}, \widetilde{T}_{2})$. 
Since $\sigma_{2}$ interchanges the cusps and preserves orientations of the longitudes, it must take $\widetilde{T}_{1}$ to $\widetilde{T}_{2}$ by a translation. Since the minimal meridional edge of $\widetilde{T}_{1}$ must map to the minimal meridional edge of $\widetilde{T}_{2}$, $\sigma_{2}$ must be the identity map between these triangulations of $\mathbb{R}^{2}$, up to deck transformations. 

\textbf{Case 3: $\Omega$ is palindromic.}
Now, we consider any additional symmetries of $Sym(M)$, which occur only if $\Omega$ is palindromic. By examining the projections of $K(\Omega)$ given in Figure \ref{palindromicsym}, we see that $\sigma_{3}$ has the following properties:
\begin{itemize}
\item $\sigma_{3}$ exchanges the maxima of $K(\Omega)$ with the minima of $K(\Omega)$.
\item $\sigma_{3}$ changes the orientation of the longitude of $K(\Omega)$ (or both longitudes if $K(\Omega)$ is a two component link).
\end{itemize} 

First, suppose that $n$ is odd. By Theorem \ref{thm:Symclass}, $\sigma_{3}$ is an orientation-preserving symmetry, and since it changes the orientation of the longitude, it must be a rotation of $\mathbb{R}^{2}$. Since $\sigma_{3}$ must exchange maximal meridional edges with minimal meridional edges, it must act as a rotation about $(\frac{1}{2},1)$ on $(\mathbb{R}^{2}, \widetilde{T})$ or rotations about $(\frac{1}{2},1)$ in both $(\mathbb{R}^{2}, \widetilde{T}_{1})$ and $(\mathbb{R}^{2}, \widetilde{T}_{2})$, if $K(\Omega)$ has two components.  

Now, suppose that $n$ is even. By Theorem \ref{thm:Symclass}, $\sigma_{3}$ is an orientation-reversing symmetry of $M$, and so, $\sigma_{3}$'s action on $\widetilde{T}$ is also orientation-reversing. Theorem \ref{SH_implies_iso} tells us that $\sigma_{3}$ must either correspond with the glide reflection $g$ or a composition of $g$ with the rotations $\rho_{1}$ and $\rho_{2}$ (up to deck transformation). This gives the desired description of $\sigma_{3}$.
\end{proof}


\section{Hidden symmetries of $2$-bridge link complements}
\label{sec:hiddensym}

Let $C(\Gamma)$ and $N(\Gamma)$ be the commensurator and normalizer of $M = \mathbb{H}^{3} / \Gamma = \mathbb{S}^{3} \setminus K(\Omega)$, respectively, as defined in Section \ref{sec:intro}. Now that we understand the symmetries of $M$ (Section \ref{sec:sym}), and the simplicial homeomorphisms of the canonical (lifted) cusp triangulation $\widetilde{T}$ (Section \ref{sec:CuspT}), we are ready to characterize the hidden symmetries of $M$, i.e., the elements of $C(\Gamma)\setminus N(\Gamma)$. Clearly, arithmetic links always have hidden symmetries, since in this case $C(\Gamma)$ is dense in $\mathrm{Isom}(\h^3)$. But hidden symmetries of arithmetic links will not necessarily be symmetries of the canonical cusp triangulation $\widetilde{T}$. We call a hidden symmetry \emph{detectable} if it is also a symmetry of $\widetilde{T}$. For non-arithmetic links, all hidden symmetries are detectable. 

Recall that $\mathrm{Aut}_{ev}(\widetilde{T})$ is the group of simplicial automorphisms of $\widetilde{T}$ preserving edge valence, so that $\mathrm{Aut}_{ev}^+(\widetilde{T})$ is the subgroup consisting of those that preserve orientation.

\subsection{Orientation-Preserving Hidden Symmetries}
\label{subsec:OPHS}

\begin{thm}
\label{thm:OPHS}
If $M = \mathbb{S}^{3} \setminus K(\Omega)$ is a hyperbolic $2$-bridge link complement, then we have the following classification of orientation-preserving hidden symmetries:
\begin{itemize}
\item If $M$ is non-arithmetic, then $M$ admits no hidden symmetries.
\item If $M$ is the figure-eight knot complement, then $M$ admits an order $6$ detectable hidden symmetry. 
\item If $M$ is the Whitehead link complement, then $M$ admits an order $4$ detectable hidden symmetry.
\item If $M$ is the $6_{2}^{2}$ link complement, then $M$ admits an order $3$ detectable hidden symmetry.
\item If $M$ is the $6_{3}^{2}$ link complement, then $M$ does not admit any detectable hidden symmetries. 
\end{itemize}
\end{thm}

\begin{proof}

\textbf{Case 1: M is non-arithmetic.} Since the triangulation $\mathcal{T}$ of $M$ is canonical, it descends to a cellulation of the minimal (orientable) orbifold $\mathcal{O}^+=\h^{3} /C^+(\Gamma)$, where $C^+(\Gamma)$ is the orientable commensurator of $M$. Hence any orientation-preserving symmetry or hidden symmetry $h\in C^+(\Gamma)\le \mathrm{Isom}^+(\h^3)$ must preserve the lifted triangulation $\widetilde{\mathcal{T}}$, which we may assume has a vertex at $\infty\in S_\infty=\R^2\cup\{\infty\}$. Since $M$ either has one cusp or has a symmetry exchanging its cusps, $N^+(\Gamma)$ acts transitively on the set of vertices of $\widetilde{\mathcal{T}}$. Thus for some $g\in N^+(\Gamma)$, $h\circ g$ fixes $\infty \in S_\infty$. Since $h$ is a symmetry of $M$ if and only if $h\circ g$ is, we may assume  that $h$ fixes $\infty \in S_\infty$. Identifying $\widetilde{T}$ with a horosphere about $\infty$, we see then that $h$ restricts to a simplicial automorphism of $\widetilde{T}$, and this restriction determines $h$ (if $K$ has two components, we understand $\widetilde{T}$ to mean a component of $\widetilde{T}_1\cup\widetilde{T}_2$). It is enough, then, to show that any element of $\mathrm{Aut}_{ev}^+(\widetilde{T})$ comes from a symmetry of $M$ (possibly composed with deck transformations of $\widetilde{T}$). 

Let $G=\Z\oplus \Z$ be the deck group of $\widetilde{T}$. By Theorem \ref{SH_implies_iso}, $\mathrm{Aut}_{ev}^+(\widetilde{T})/G$ is generated by $\{\rho_1, \rho_2, \rho_3\}$ if $\rho_3$ is a simplicial automorphism, and is generated by  $\{\rho_1, \rho_2\}$ if $\rho_3$ is not simplicial. 
Let $\sigma_1$, $\sigma_2$, $\sigma_3$ be the symmetries described in Proposition \ref{syms}, 
and let $H$ be the horizontal strip in the first quadrant with a vertex at the origin.


We first observe that $\rho_1=\sigma_1$, and $\rho_2$ is either $\sigma_2$, or $\sigma_1$ composed with a deck transformation, depending on whether $K$ has $1$ or $2$ components. Hence $\rho_1$ and $\rho_2$ come from symmetries of $M$ in both cases, and so for the case where $\rho_3$ is not simplicial, $M$ cannot have hidden symmetries. If $\rho_3$ \emph{is} simplicial, then since the reflection $r_y$ across $y=1$ is always a simplicial automorphism (by construction of $\widetilde{T}$), the reflection $\rho_3\circ r_y$ across $x=\frac{1}{2}$ is also simplicial. Hence in this case $H$ is symmetric about the line $x=\frac{1}{2}$, and so $\Omega$ is palindromic with $\Omega_c=R$, and it follows that $\rho_3$ comes from the symmetry $\sigma_3$ of $M$. Again, we conclude that $M$ has no hidden symmetries.

\textbf{Case 2: M is arithmetic.} There are exactly $4$ arithmetic $2$-bridge links: the figure-eight knot ($\Omega=RL$), the Whitehead link ($\Omega=RLR$), the $6_2^2$ link ($\Omega=R^2L^2$), and the $6_3^2$ link ($\Omega=RL^2R$).

Since $\Omega=RL^2R$ is not an excluded case in Lemma \ref{SH_lemma} and its corollaries, the arguments in Case 1 above show that, if $M$ is the $6_3^2$ link complement, then every $h\in\mathrm{Aut}_{ev}^+(\widetilde{T})$  that preserves edge valence comes from a symmetry of $M$, i.e., $M$ admits no detectable orientation-preserving hidden symmetries. 

If $M$ is the figure-eight knot, the Whitehead link complement, or the $6_2^2$ link complement, then we can see by edge/vertex (valence) correspondences in $\widetilde{T}$ that if $e$ and $e'$ are two edges of a tetrahedron in $\mathcal{T}$ which are opposite each other (i.e., they do not share a vertex), then $val(e)=val(e')$. This is evident in $\widetilde{T}$ by the fact that any edge and vertex of $\widetilde{T}$ that are across from each other (i.e., their convex hull is a single triangle of $\widetilde{T}$) have the same valence. This makes it easy to identify the (unique) hyperbolic structure on $\mathcal{T}$. If an edge of a tetrahedron has valence $k$, then we make the dihedral angle at that edge $\frac{2\pi}{k}$. We just need to make sure that this gives a Euclidean structure to the cusp cross sections, but this is confirmed by Figure \ref{fig:tilings}. It follows that the depictions of $\widetilde{T}$ in Figure \ref{fig:tilings} are actually metrically correct (up to scaling), so the rotations $\rho_v$ indicated are isometries of $\widetilde{T}$. Next we check that $\rho_v$ extends to an isometry of the 3-dimensional triangulation $\widetilde{\mathcal{T}}$. Viewing $\widetilde{T}$ as a horosphere about $\infty$ in the upper half-space model of $\h^3$, the vertex $v$ about which $\rho_v$ rotates $\widetilde{T}$ corresponds to some edge $e_v$ of $\widetilde{\mathcal{T}}$ connecting $\infty$ to a point $p_v \in \partial \h^3\setminus \{\infty\}$. The rotation of $\h^3$ about $e_v$ that agrees with $\rho_v$ on $\widetilde{T}$ induces a rotation of the lift $\widetilde{T}_v$ of $T$ centered at $p_v$, which is an isometry since $\widetilde{T}$ and $\widetilde{T}_v$ are isometric and $e_v$ appears in both as a vertex of the same valence. If $v_1$ is some other vertex of $\widetilde{T}$, and $\rho_v(v_1)=v_2$, then since $\rho_v$ differs from $\rho_{v_1}$ by composition with symmetries of $M$ and deck transformations of $\widetilde{T}$, the rotation of $\h^3$ induced by $\rho_v$ takes $\widetilde{T}_{v_1}$ to $\widetilde{T}_{v_2}$ isometrically. It follows that $\rho_v$ induces an isometry on $\widetilde{\mathcal{T}}$, of the order indicated in the statement of the theorem.
\end{proof}

\begin{figure}
        \centering
        \begin{subfigure}[b]{0.32\textwidth}
                \includegraphics[scale=0.6]{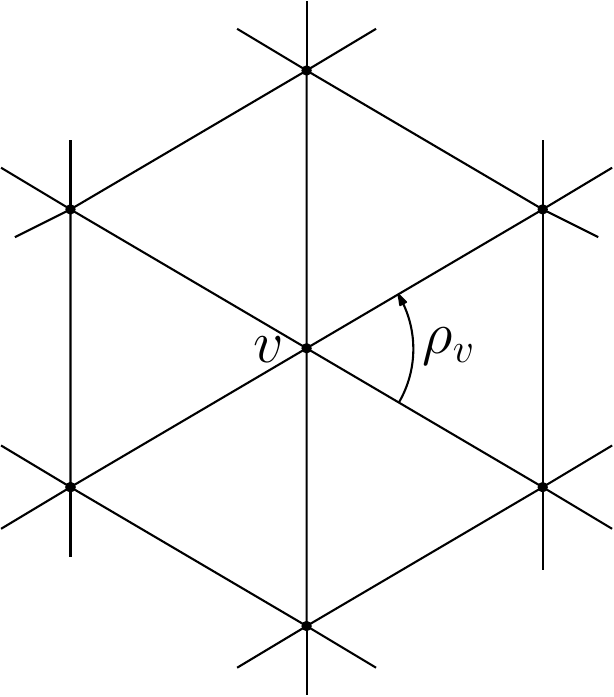}
                \caption{$\Omega=RL$}
                \label{fig:tilings_a}
        \end{subfigure}
       \begin{subfigure}[b]{0.32\textwidth}
                \includegraphics[scale=0.6]{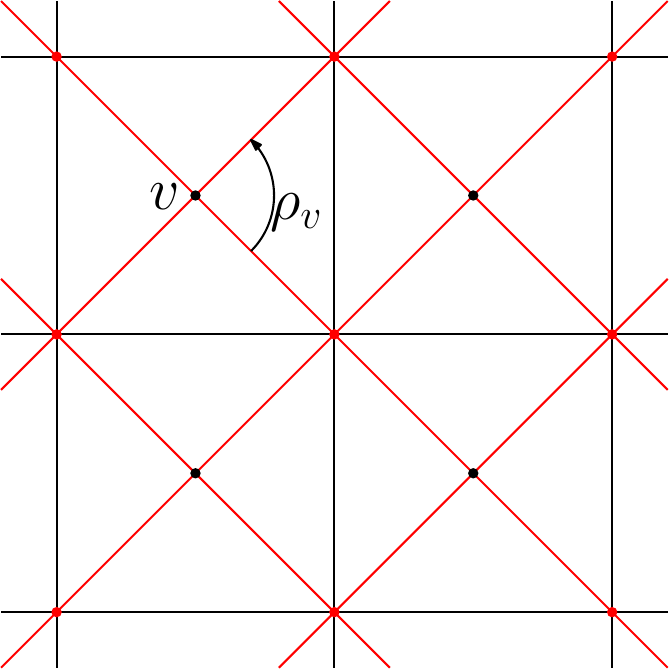}
                \caption{$\Omega=RLR$}
                \label{fig:tilings_b}
        \end{subfigure}     
        \begin{subfigure}[b]{0.32\textwidth}
                \includegraphics[scale=0.7]{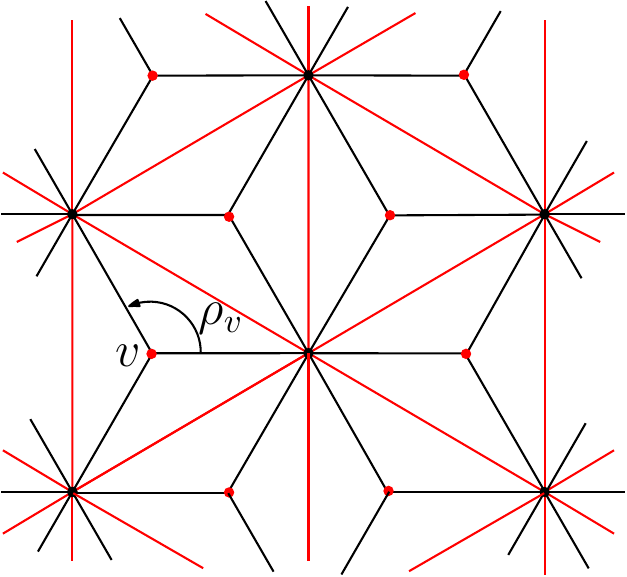}
                \caption{$\Omega=R^2L^2$}
                \label{tilings_c}
        \end{subfigure}  
        \caption{Lifted cusp triangulation $\widetilde{T}$ for the figure-eight knot, Whitehead link, and $6^2_2$ link complements, from left to right. Edges/vertices with the same coloring (within each figure) have the same valence.}
        \label{fig:tilings}
\end{figure}


\subsection{Orientation-Reversing Hidden Symmetries}
\label{subsec:ORHS}

\begin{thm}
\label{thm:ORHS}
If $M = \mathbb{S}^{3} \setminus K(\Omega)$ is a hyperbolic $2$-bridge link complement, then we have the following classification of orientation-reversing hidden symmetries: 
\begin{itemize}
\item If $M$ is non-arithmetic, then $M$ admits no orientation reversing hidden symmetries.
\item If $M$ is the $6_{3}^{2}$ link complement, then $M$ admits no detectable orientation reversing hidden symmetry.
\item If $M$ is the figure-eight knot complement, the Whitehead link complement, or the $6_{2}^{2}$ link complement, then $M$ admits an order $2$ orientation reversing hidden symmetry. 
\end{itemize}
\end{thm}

\begin{proof}
\textbf{Case 1: M is non-arithmetic.} The proof will be analogous to the orientation preserving case. As in that case, we need only show that any $h\in\mathrm{Aut}_{ev}(\widetilde{T})$ is in fact a symmetry of $M$. By Theorem \ref{SH_implies_iso}, $h$ must be a composition of $\rho_1,\rho_2,\rho_3$, and $g$, where $\rho_1$, $\rho_2$, and $\rho_3$ are the rotations by $\pi$ about $(1,1)$, $(2,1)$, and $(\frac{1}{2},1)$, respectively, and $g$ is the glide reflection given by the composition of $r_{\frac{1}{2}}$ with $(x,y)\mapsto (x,y+1)$. If $g\notin \mathrm{Aut}_{ev}(\widetilde{T})$, then $\mathrm{Aut}_{ev}(\widetilde{T})=\mathrm{Aut}_{ev}^+(\widetilde{T})$, and we are done. If $g\in\mathrm{Aut}_{ev}(\widetilde{T})$, then it is clear from the construction of $\widetilde{T}$ that we must have $\Omega_c=L$, and $\Omega$ must be palindromic. In this case, though, $g$ corresponds to the symmetry $\sigma_3$ in the notation of Proposition \ref{syms}, so the non-arithmetic case is proved.

\textbf{Case 2: M is arithmetic.} The proof is analogous to the orientation preserving case.

\end{proof}


\subsection{Irregular Coverings By Hyperbolic $2$-Bridge Link Complements}
\label{subsec:irregular}
	
Theorem \ref{thm:OPHS} and Theorem \ref{thm:ORHS} give us the following corollary about irregular coverings of $3$-manifolds.  
	
\begin{cor}
\label{cor:coverings}
Let $N$ be any hyperbolic $2$-bridge link complement. If $N$ is non-arithmetic, then $N$ does not irregularly cover any hyperbolic $3$-orbifolds (orientable or non-orientable). If $N$ is arithmetic, then $N$ does not irregulary cover any (orientable) hyperbolic $3$-manifolds. 
\end{cor}
	
\begin{proof}
By Theorem \ref{thm:OPHS} and Theorem \ref{thm:ORHS}, any non-arithmetic hyperbolic $2$-bridge link complement $N$ does not have any hidden symmetries (orientation-preserving or orientation-reversing). Thus, if any such $N$ covers a hyperbolic $3$-orbifold, it must be a regular cover. 
		
If $N$ is arithmetic, then $N$ is the complement of either the figure-eight knot, the Whitehead link, the $6_{2}^{2}$ link, or the $6_{3}^{2}$ link. If $N$ irregularly covers some hyperbolic $3$-manifold $N'$, then it must be at least a degree $3$ covering. Here, we get a volume contradiction. Cao--Meyerhoff in \cite{CaMe} showed that the figure-eight knot complement and its sister are the orientable cusped hyperbolic $3$-manifold of minimal volume, with volume $\geq 2.029$. Thus, $vol(N') \geq 2.029$, and so, $vol(N) \geq 3(2.029) = 6.087$. However, the volumes of any of the four arithmetic hyperbolic $2$-bridge link complements are strictly smaller than $6.087$. Thus, we can't have any such irregular coverings in the arithmetic case.
\end{proof}

In \cite{BoWe}, Boileau--Weidmann give a characterization of $3$-manifolds that admit a nontrivial JSJ-decomposition and whose fundamental groups are generated by two elements. Their work shows that there are four possibilities for such manifolds, one of which is that the hyperbolic part of the JSJ decomposition admits a finite-sheeted irregular covering by a hyperbolic $2$-bridge link complement. Corollary \ref{cor:coverings} immediately eliminates this possibility, giving the following revised characterization of such manifolds. In the following corollary, $D$ stands for a disk, $A$ for an annulus, and $Mb$ for a M\"{o}bius band. For an orbifold, cone points are listed in parentheses after the topological type of the orbifold is given.     
	
\begin{cor}
\label{cor:JSJrank}
Let $M$ be a compact, orientable, irreducible $3$-manifold with rank($\pi_{1}(M)) = 2$. If $M$ has a non-trivial JSJ-decomposition, then one of the following holds:
		
1. $M$ has Heegaard genus $2$.
		
2. $M = S \cup_{T} H$ where $S$ is a Seifert manifold with basis $D(p,q)$ or $A(p)$, $H$ is a hyperbolic manifold and $\pi_{1}(H)$ is generated by a pair of elements with a single parabolic element. The gluing map identifies the fiber of $S$ with the curve corresponding to the parabolic generator of $\pi_{1}(H)$. 
		
3. $M= S_{1} \cup_{T} S_{2}$ where $S_{1}$ is a Seifert manifold over $Mb$ or $Mb(p)$ and $S_{2}$ is a Seifert manifold over $D(2,2l+1)$. The gluing map identifies the fiber of $S_{1}$ with a curve on the boundary of $S_{2}$ that has intersection number one with the fiber of $S_{2}$.

\end{cor}


\section{Commensurability of $2$-bridge link complements}
\label{sec:comm}

In this section, we show that there is only one pair of commensurable hyperbolic $2$-bridge link complements. We accomplish this by analyzing the cusp of the unique minimal orbifold in the commensurability class of a non-arithmetic hyperbolic $2$-bridge link complement. 

Let $M = \mathbb{S}^{3} \setminus K(\Omega) = \mathbb{H}^{3} / \Gamma$ be any non-arithmetic hyperbolic $2$-bridge link complement. By a theorem of Margulis \cite{Mar}, there exists a unique minimal (orientable) orbifold in the commensurability class of $M$, specifically, $\mathcal{O}^+ = \mathbb{H}^{3}/C^{+}(\Gamma)$. By Theorem \ref{thm:OPHS} we know that $M$ admits no hidden symmetries, and so, $C^{+}(\Gamma) = N^{+}(\Gamma)$. Since $N^{+}(\Gamma) /\Gamma = Sym^{+}(M)$, we only have to quotient $M$ by its orientation-preserving symmetries to obtain $\mathcal{O}^+$. 

We will analyze the commensurability class of $M$ by considering the cusp of $\mathcal{O}^+$. Recall that every $2$-bridge link is either a knot or a link with two components. If $K$ has two components, then there always exists a symmetry exchanging those components; see Section \ref{sec:sym}. Thus, the orbifold $\mathcal{O}^+$ admits a single cusp, $C$. If we quotient the cusp(s) of $M$ along with the cusp triangulation $T$ by the symmetries of $M$, then we obtain  the cusp $C$ of $\mathcal{O}^+$, along with a canonical cellulation, $T_{C}$. Technically, $T_{C}$ is not a triangulation, but just a quotient of a triangulation (hence we call it a cellulation). If $M$ and $M'$ are commensurable, then their corresponding minimal orbifolds, must admit isometric cusps that have identical cusp triangulations. In this case, we say that the corresponding cusp cellulations, $T_{C}$ and $T_{C'}$, are \textit{equivalent}. We wish to determine when these cusps are equivalent. The following two lemmas takes care of this classification.

\begin{lemma}
\label{lemma:cusp1}
Let $M = \mathbb{S}^{3} \setminus K(\Omega)$ be a non-arithmetic hyperbolic $2$-bridge link complement. Suppose $\Omega$ is not palindromic or $n$ is even. Then $C \cong S^{2}(2,2,2,2)$ and $T_{C}$ determines the word $\Omega$ up to inversion and switching Ls and Rs. 
\end{lemma}

\begin{proof}
By Theorem \ref{thm:Symclass}, $Sym^{+}(M) \cong Z_{2} \oplus Z_{2}$, and Proposition \ref{syms} tells us exactly how $Sym^{+}(M)$ acts on $T$ and $\widetilde{T}$. First, assume $K(\Omega)$ is a knot. Here, we choose the rectangle $\left[0,4\right] \times \left[0,2\right]$ in $\widetilde{T}$ as a fundamental domain for the torus $T$. In this case, $\sigma_{1} \circ \sigma_{2}$ acts as a translation of $\widetilde{T}$ by $(x,y) \rightarrow (x+2, y)$.  When we quotient our fundamental domain by the symmetry $\sigma_{1} \circ \sigma_{2}$, we produce a fundamental domain for a torus given by the rectangle $\left[0,2\right] \times \left[0,2\right]$, with opposite sides identified. If $K(\Omega)$ is a link with two components, then our fundamental domain for $T$ is given by two copies of $\left[0,2\right] \times \left[0,2\right]$. When we quotient by $\sigma_{2}$, we just exchange the cusps. This again produces a fundamental domain for a (single) torus of the form $\left[0,2\right] \times \left[0,2\right]$ in $\widetilde{T}$. In either case (a knot or a two component link), we just need to quotient by $\sigma_{1}$, which acts as a rotation about $(1,1)$, to obtain $C$ along with $T_{C}$. This gives us a fundamental domain of the form $\left[0,1\right] \times \left[0,2\right]$, with identifications given in Figure \ref{fig:orbifoldcusp1}. We can see that this resulting cusp is $S^{2}(2,2,2,2)$. 

\begin{figure}
\includegraphics[scale=0.7]{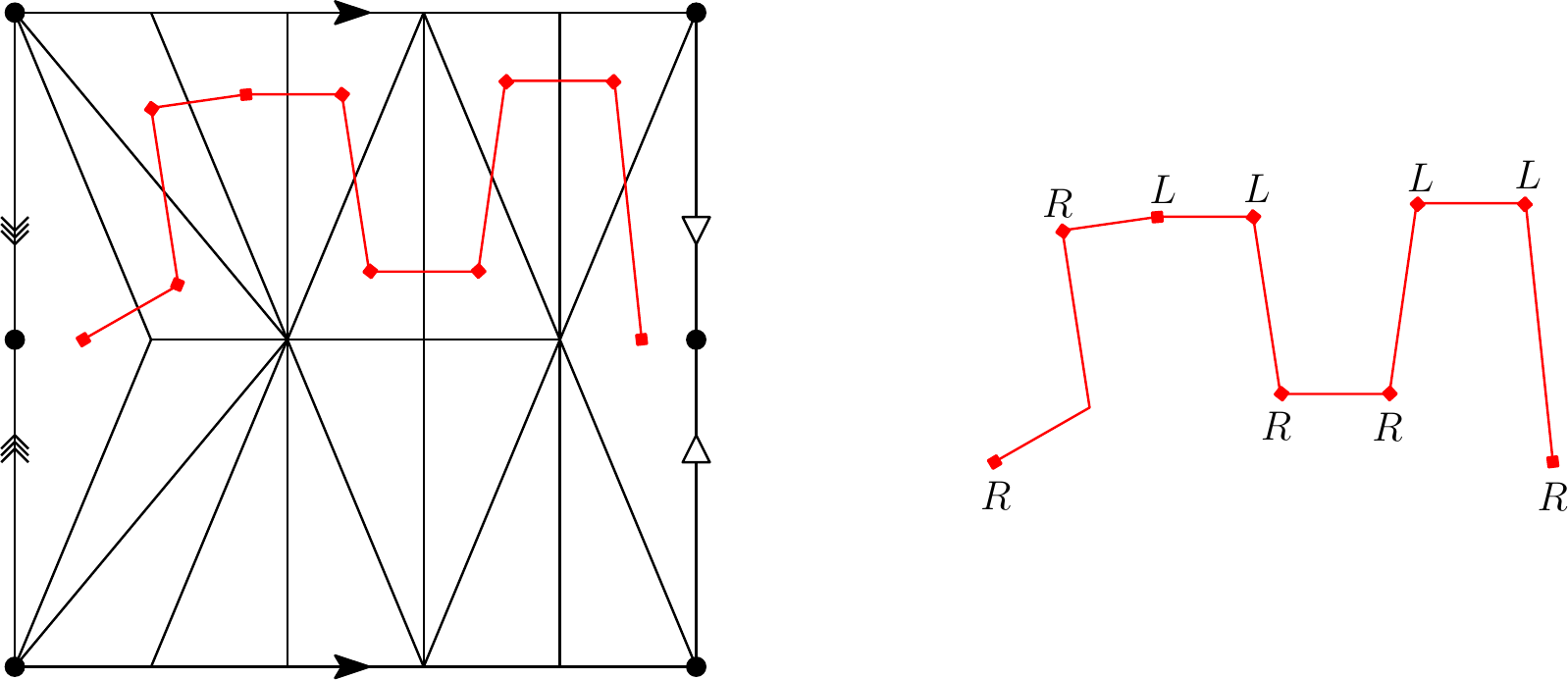}
\caption{The cusp triangulation $T_{C}$ for the word $\Omega = R^{2}L^{3}R^{2}L^{2}R$. The order two singularities are marked by solid black circles. The red line segment gives $l_{C}$.}
\label{fig:orbifoldcusp1}
\end{figure}

To each such $T_{C}$ we associate a labeled line segment, $l_{C}$, in the following manner depicted in Figure \ref{fig:orbifoldcusp1}. The two  endpoints of this line segment come from vertices placed in the centers of the two clasping triangles of the fundamental domain of $T_{C}$. We also place a vertex in the center of each triangle in the top half of the triangulation of the fundamental domain for $T_{C}$. We connect two vertices by an edge if and only if the corresponding triangles in $T_{C}$ share an edge. We label each vertex of $l_{C}$ (including the endpoints) by $L$ or $R$ corresponding to the label of the triangle in $T_{C}$. We say that $l_{C}$ is equivalent to another labeled line segment $l_{C'}$ if there exists a simplicial homeomorphism between the two that preserves labelings or switches Ls and Rs between labelings. 

Now, $T_{C}$ is equivalent to $T_{C'}$ if and only if $l_{C}$ is equivalent to $l_{C'}$. This holds because $l_{C}$ tells you exactly how to build $T_{C}$ and vice versa. However, there are only two possibilities for how $l_{C}$ can be equivalent to $l_{C'}$: either the left endpoint maps to the left endpoint, or the left endpoint maps to the right endpoint. In the first case, $\Omega$ must be the same as $\Omega'$. In the second case, $\Omega'$ must be an inversion of $\Omega$.    
\end{proof}

\begin{lemma}
\label{lemma:cusp2}
Let $M = \mathbb{S}^{3} \setminus K(\Omega)$ be a non-arithmetic hyperbolic $2$-bridge link complement. Suppose $\Omega$ is palindromic and $n$ is odd. Then $C \cong S^{2}(2,2,2,2)$ and $T_{C}$ determines the word $\Omega$ up to inversion and switching Ls and Rs. 
\end{lemma}

\begin{proof}
By Theorem \ref{thm:Symclass}, either $Sym^{+}(M) \cong Z_{2} \oplus Z_{2} \oplus Z_{2}$, or $Sym^{+}(M) \cong D_{4}$. Just as in the previous lemma, we can first quotient a fundamental domain for $T$ in $\widetilde{T}$ by the $Z_{2} \oplus Z_{2}$ subgroup of $Sym^{+}(M)$ to obtain a single $S^{2}(2,2,2,2)$ cusp. To obtain $C$ and $T_{C}$, we also quotient by the action of $\sigma_{3}$, which is a rotation about $(\frac{1}{2}, 1)$ in $\widetilde{T}$ by Proposition \ref{syms}; see Figure \ref{fig:orbifoldcusp2}.   

Similar to Lemma \ref{lemma:cusp1}, we can associate a marked line segment $l_{C}$ to each cusp $T_{C}$, as depicted in Figure \ref{fig:orbifoldcusp2}. Once again, we see that this marked line segment determines $T_{C}$ up to inversions and switching Ls and Rs. We leave the details for the reader.
\end{proof}

\begin{figure}[h]
\includegraphics[scale=0.7]{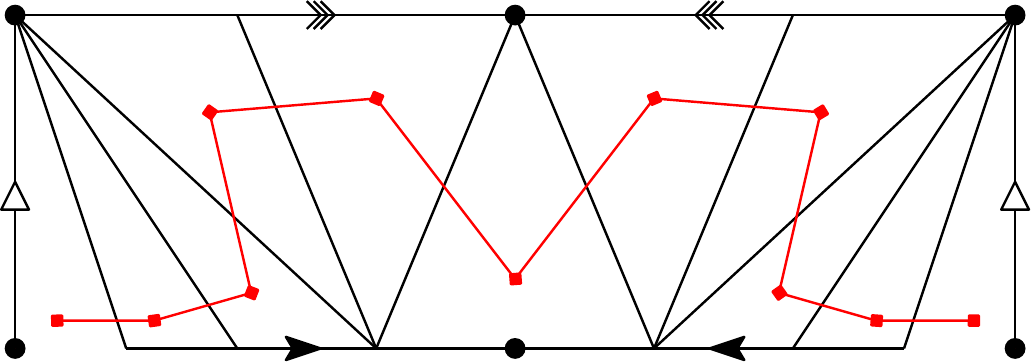}
\caption{The cusp triangulation $T_{C}$ for the word $\Omega = R^{3}L^{2}R^{}L^{2}R^{3}$. The order two singularities are marked by solid black circles. The red line segment gives $l_{C}$.}
\label{fig:orbifoldcusp2}
\end{figure}

\begin{cor}
\label{cor:cusp}
Let $M = \mathbb{S}^{3} \setminus K(\Omega)$ be a non-arithmetic hyperbolic $2$-bridge link complement. Then $C \cong S^{2}(2,2,2,2)$ and $T_{C}$ is determines the word $\Omega$ up to inversion and switching Ls and Rs.
\end{cor}
\begin{proof}
We claim that the two types of cusp cellulations coming from Lemma \ref{lemma:cusp1} and Lemma \ref{lemma:cusp2} can not be equivalent. First, note that the tiling $T_{C}$ for an $S^{2}(2,2,2,2)$ from Lemma \ref{lemma:cusp1} always has singularities located at vertices. Furthermore, any of these vertices with singularities have valence $\neq 2$. Now, the tiling coming from Lemma \ref{lemma:cusp2} either has a singularity that is not located at a vertex (this happens if $\alpha_{\frac{n+1}{2}}$ is odd) or it has a singularity located at a vertex of valence $2$ (this happens if $\alpha_{\frac{n+1}{2}}$ is even). Thus, these two types of cusp cellulations can not be equivalent, and so, the previous two lemmas imply that any such $T_{C}$ is determined by the word $\Omega$ up to inversion and switching Ls and Rs. 
\end{proof}

We can now prove our main theorem. 

\begin{thm}
\label{thm:comm}
The only commensurable hyperbolic $2$-bridge link complements are the figure-eight knot complement and the $6_{2}^{2}$ link complement. 
\end{thm}

\begin{proof}
It is a well known fact that cusped, arithmetic hyperbolic $3$-manifolds are commensurable if and only if they have the same invariant trace field; see \cite{MaRe} for details. The figure-eight knot complement and the $6_{2}^{2}$ link complement both have invariant trace field $\mathbb{Q}(\sqrt{-3})$, while the Whitehead link complement has $\mathbb{Q}(\sqrt{-1})$ and the $6_{3}^{2}$ link complement has $\mathbb{Q}(\sqrt{-7})$. Thus, among hyperbolic arithmetic $2$-bridge link complements, only the figure-eight knot complement and the $6_{2}^{2}$ link complement are commensurable. Now, a non-arithmetic hyperbolic $2$-bridge link complement can not be commensurable with an arithmetic hyperbolic $2$-bridge link complement. This is because their commensurators determine their commensurablilty classes, and by a theorem of Margulis \cite{Mar}, the commensurator of a hyperbolic $3$-manifold is discrete if and only if it is non-arithmetic. 

It remains to check that non-arithmetic hyperbolic $2$-bridge link complements are pairwise incommensurable. Let $M = \mathbb{S}^{3} \setminus K(\Omega)$ and $M' = \mathbb{S}^{3} \setminus K(\Omega')$ be any two such manifolds. We use $T_{C}$ and $T_{C'}$ to denote the cusp cellulations of the minimal orbifolds in the commensurability classes of $M$ and $M'$ respectively. Recall that if $T_{C}$ is not equivalent to $T_{C'}$, then $M$ and $M'$ are not commensurable. By Corollary \ref{cor:cusp}, $T_{C}$ and $T_{C'}$ are equivalent only if $\Omega$ and $\Omega'$ differ by inversion or switching Ls and Rs. As noted in Section \ref{sec:background}, both of these possibilities result in $M$ and $M'$ being isometric. Thus, $M$ and $M'$ are commensurable only if they are isometric, as desired.
\end{proof}


\bibliographystyle{hamsplain}
\bibliography{biblio}

\end{document}